\documentclass[article]{amsart} 

\usepackage[foot]{amsaddr}

\usepackage{mathpazo}

\usepackage[mathcal]{eucal}
\DeclareFontFamily{OT1}{pzc}{}
\DeclareFontShape{OT1}{pzc}{m}{it}{<-> s * [1.10] pzcmi7t}{}
\DeclareMathAlphabet{\mathpzc}{OT1}{pzc}{m}{it}

\usepackage{wrapfig}
\usepackage{hyperref}
\usepackage{amssymb}
\usepackage{mathrsfs}  
\usepackage{latexsym}
\usepackage{amsmath}
\usepackage{color}
\usepackage{enumitem}
\usepackage{bbm} 
\usepackage{tikz}
\usepackage{tikz-cd}

\usepackage{mathtools}
\newcommand{\defeq}{\vcentcolon=}
\def\co{\colon\thinspace} 

\mathchardef\mhyphen="2D

\usepackage{academicons}
\usepackage{xcolor}

\newcommand{\orcid}[1]{\href{https://orcid.org/#1}{\textcolor[HTML]{A6CE39}{\aiOrcid}}}

\numberwithin{equation}{subsection}

\newtheorem{thmA}{Theorem}

\newtheorem{corA}[thmA]{Corollary}

\newtheorem{theorem}{Theorem}[subsection]  
\newtheorem{lemma}[theorem]{Lemma} 
\newtheorem{proposition}[theorem]{Proposition}
\newtheorem{corollary}[theorem]{Corollary}

\newtheorem{example}[theorem]{Example}

\theoremstyle{remark} 
\newtheorem{definition}[theorem]{Definition}
\newtheorem{remark}[theorem]{Remark}

\newcommand{\spec}{\mathrm{Spec}\:}

\newcommand{\trop}{\mathpzc{trop}}

\newcommand{\bend}{\mathpzc{B}}

\newcommand{\Hom}{\mathrm{Hom}}

\DeclareMathOperator*{\colim}{colim} 
\newcommand{\sslash}{\mathbin{/\mkern-6mu/}}

\newcommand{\semirings}{\mathbf{\mathrm{Semirings}}}
\newcommand{\Dsemirings}{\mathbf{\mathrm{DiffSemirings}}}

\newcommand{\pairs}{\mathbf{\mathrm{Pairs}}}
\newcommand{\reducedpairs}{\mathbf{\mathrm{Pairs}}_\mathrm{red}}

\usepackage{todonotes}

\title{A general framework for tropical differential equations}

\author{Jeffrey Giansiracusa}
\email{jeffrey.giansiracusa@durham.ac.uk} 
\address{Department of Mathematical Sciences, Durham University, Upper Mountjoy Campus, Stockton Road, Durham DH1 3LE, UK. \\ ORCiD: 0000-0003-4252-0058}

\author{Stefano Mereta*}
\email{Stefano.Mereta@mis.mpg.de}
\address{ Max Planck Institute for Mathematics in the Sciences, Inselstraße 22, 04103 Leipzig, Germany \\ ORCiD: 0000-0002-4549-2534}

\date{\today}


\begin{document}
\begin{abstract}
We construct a general framework for tropical differential equations based on idempotent semirings
and an idempotent version of differential algebra. Over a differential ring equipped with a
non-archimedean norm enhanced with additional differential information, we define tropicalization of
differential equations and tropicalization of their solution sets. This framework includes rings of
interest in the theory of $p$-adic differential equations: rings of convergent power series over a
non-archimedean normed field.  The tropicalization records the norms of the coefficients. This gives
a significant refinement of Grigoriev's framework for tropical differential equations. We then prove
a differential analogue of Payne's inverse limit theorem: the limit of all tropicalizations of a
system of differential equations is isomorphic to a differential variant of the Berkovich
analytification.
\end{abstract}
\maketitle

\section{Introduction}

Algebraic ODEs are systems of differential equations formed from polynomial expressions in an
indeterminate function $f$ and its derivatives.  The algebraic theory was first developed by Ritt
\cite{Ritt} and collaborators.  Many important classes of models from the natural sciences, such as
chemical reaction networks, are algebraic ODEs, and in pure mathematics algebraic ODEs appear in
many parts of geometry, including periods and monodromy. Understanding their solutions and
singularities has many important consequences in pure and applied mathematics.

In this paper we pursue the further development of the tropical mathematics tool set for studying
differential equations.  Tropical geometry is a field at the interface between combinatorics,
computational algebra/geometry, and algebraic geometry. One of the foci of tropical geometry is the
study of non-archimedean amoebae of affine varieties over non-archimedean normed fields, which can
be viewed as combinatorial shadows of varieties. Tropical geometry has provided important
computational and theoretical tools for algebraic geometry, and we hope to open the door to similar
applications in differential algebra.

In \cite{Grigoriev-TDE} Grigoriev first introduced a theory of tropical differential equations and
defined a framework for tropicalizing algebraic ODEs over a ring of formal power series $R[[t]]$.
In this framework, one tropicalizes a differential equation by recording the leading power of $t$
in each coefficient, and one tropicalizes a power series solution simply recording the powers of $t$ that
are present.  

Solutions to a differential equation tropicalize to solutions to its tropicalization, and Grigoriev
asked if \emph{all} solutions to the tropicalization of an equation arise as tropicalizations of
classical solutions; i.e., is the map from classical solutions to tropical solutions surjective?
This is the differential equation analogue of the Fundamental Theorem of Tropical Geometry
\cite[Theorem 3.2.3]{Maclagan-Sturmfels}, and this question was answered positively by Aroca et
al.~in \cite{AGT-fundamental-theorem} (assuming $R$ is an uncountable algebraically closed field of
characteristic 0). These ideas have also been extended to the case of algebraic \emph{partial}
differential equations in \cite{Falkensteiner-Fundamental-theorem}.

Paralleling the role of Gr{\"o}bner theory in defining tropical varieties in the non-differential
setting, \cite{Fink-differential-initial-forms} and \cite{Hu-tropical-groebner} define initial forms
and develop a Gr{\"o}bner-theoretic approach to Grigoriev's tropical differential equations.  A similar
approach is also presented in \cite{Cotterill-DE}, which also gives an illuminating account of
tropical ordinary and partial differential equations (in part based on a preliminary report of the
algebraic perspective presented here).

A limitation present in all of the above work is that the tropicalization construction studied there
records only the powers of $t$ present in a power series solution; it does not record any
information about the norms of the coefficients. Thus any information about convergence of power
series solutions is lost when using Grigoriev's tropicalization.  In the theory of $p$-adic
differential equations, one of the central problems is to understand the radii of convergence of
formal power series solutions, which are controlled by the norms of the coefficients.  

\subsection{Results}
The main purpose of this work is to build a refinement of Grigoriev's framework that records and
incorporate the norms of the coefficients in a power series solution so that convergence information
is encoded in tropical solutions.  This requires developing a theory of differentials on idempotent
semirings in which the usual Leibniz rule is weakened to a tropical Leibniz rule, and this
development includes constructing free tropical differential algebras (a tropical analogue of Ritt
algebras).

We give a brief explanation of our framework here. A \emph{tropical pair} $\mathbf{S}=(S_1 \to S_0)$
is a tropical differential semiring $S_1$ and a homomorphism to a semiring $S_0$. The coefficients
of tropical differential equations live in $S_0$.  Solutions live in $S_1$ (where they can be
differentiated), but the condition that tests if something is a solution takes place in $S_0$.  We
think of $S_0$ as recording the leading behavior of elements of $S_1$.  The primary example of a
tropical pair has $S_1 = \mathbb{T}[[t]]$ (the semiring of formal power series with tropical real
number coefficients), $S_0 = \mathbb{R}^2_\mathit{lex} \cup \{\infty\}$ is a rank 2 version of the
tropical semiring, and the map $S_1 \to S_0$ sends $at^n + \cdots$ to $(n,a)$.

We now state out main results informally. 



\begin{thmA}
We construct a category of $\mathbf{S}$-algebras, and to a set $E$ of tropical differential equations over
$\mathbf{S}$ we associate an object of this category such that morphisms to an $\mathbf{S}$-algebra
$\mathbf{T}$ are in natural bijection with solutions to $E$ with values in $\mathbf{T}$.
\end{thmA}

A system of algebraic differential equations over a field $k$ is represented in coordinate-free form
by a differential $k$-algebra $A$.  To tropicalize $A$, we need two pieces of additional data:
\begin{enumerate}
\item A non-archimedean norm on $k$ taking values in an idempotent semiring $S_0$, and a
\emph{differential enhancement} of the norm, which is a lifting to a map $A \to S_1$ that commutes with the
differential.  (These notions are defined in Sections \ref{sec:seminorms} and
\ref{sec:differential-enhancement}.)
\item A system of generators $x_i \in A$ so that $A$ is presented as a quotient of a Ritt algebra
$k\{x_1, \ldots, x_n\} \twoheadrightarrow A$.
\end{enumerate}

Any differential algebra $A$ admits a universal presentation $k\{x_a \: | \: a\in A\} \to A$.
Tropicalizing this presentation, we find:

\begin{thmA}\label{thmn:colimit}
The tropicalization of $A$ with respect to its universal presentation is the colimit of its
tropicalizations with respect to finite presentations.  
\end{thmA}

Finally, we provide evidence for the appropriateness of our definitions and framework by proving a
differential analogue of Payne's inverse limit theorem \cite{Payne}.  Recall that, given an algebra
$A$ over a non-archimedean field $k$, the underlying set of the Berkovich analytification of $\spec
A$ is the set of all multiplicative seminorms on $A$ that are compatible with the norm on $k$.  Now
suppose that $k$ is a differential ring, the norm $v$ on $k$ has a differential enhancement
$\widetilde{v}$ taking values in a pair $\mathbf{S}$, and $A$ is a differential algebra over $k$. In
this setting, given an $\mathbf{S}$-algebra $\mathbf{T}=(T_1 \to T_0)$, we can consider the set of
all pairs $(w,\widetilde{w})$ where $w\co A\to T_0$ is a multiplicative seminorm on $A$ compatible
with $v$ and $\widetilde{w}\co A \to T_1$ is a differential enhancement of $w$ compatible with
$\widetilde{v}$. We call this the $\mathbf{T}$-valued \emph{differential Berkovich space of $A$},
denoted $\mathit{Berk}_{\mathbf{T}}(A)$.

\begin{thmA}
There is a universal valuation with differential enhancement on $A$, and it takes values in the
tropicalization of the universal presentation of $A$.  Hence the tropicalization of the universal
presentation corepresents the functor $\mathbf{T}
\mapsto \mathit{Berk}_{\mathbf{T}}(A)$.
\end{thmA}

Combining this with Theorem \ref{thmn:colimit}, we immediately obtain our differential analogue of
Payne's inverse limit theorem.

\begin{corA}
Let $k$ be a differential ring equipped with a non-archimedean seminorm and differential enhancement
taking values in $\mathbf{S}$, let $A$ be a differential algebra over $k$, and let $\mathbf{T}$ be
an $\mathbf{S}$-algebra. The set $\mathit{Berk}_{\mathbf{T}}(A)$ is isomorphic to the inverse limit
of the $\mathbf{T}$-valued solution sets of the tropicalizations of all finite presentations of $A$.
\end{corA}

\subsection*{Acknowledgements}
We thank Alex Fink and Zeinab Toghani for many useful discussions and for organizing an influential
workshop at QMUL in December 2019 in which the tropical differential equations community came
together and shared many ideas.  We also thank Andrea Pulita for his encouragement to do this.  S.M.
gratefully acknowledges the support of a Swansea University Strategic Partnerships Scholarship and
funding from Université Grenoble Alpes IDEX.

\section{Semirings and non-archimedean seminorms}

Several algebraic foundations for tropical geometry have been developed, including hyperfields
\cite{Viro, Baker-Bowler,Lorscheid-hyperfield,Jun-hyperfields} Lorscheid's blueprints
\cite{Lorscheid-blueprints1,Lorscheid-scheme-theoretic}, and idempotent semirings
\cite{GG1,GG2,GG3,Maclagan-Rincon-1,Maclagan-Rincon-2,Maclagan-Rincon-3,Mincheva-Joo-1,Mincheva-Joo-2,Betram-Easton,Yaghmayi,Noah-module-theoretic}.
For the present work, we find that idempotent semirings provide the most convenient language for the
development of our theory.

\subsection{Idempotent semirings}
An idempotent semiring is a semiring $(S, \oplus, \otimes)$ in which addition is an idempotent
operation:  $a \oplus a = a$.  An idempotent semiring carries a canonical partial order defined by
$a \leq b$ if $a\oplus b = b$.  The additive unit $0_S$ is the unique minimal element.  
In a semiring, we will often write the product $a\otimes b$ simply as $ab$.

\begin{example}
\begin{enumerate}
\item Let $\mathbb{T}=\mathbb{R}_{\geq 0}$ with the operations $a\otimes b = ab$  (ordinary multiplication), and 
$a \oplus b =   \mathrm{max}(a,b)$.    Note that the map $a\mapsto -\mathrm{log}(a)$ gives an
isomorphism to the usual tropical semiring $(\mathbb{R}\cup \{\infty\}, \mathrm{min},+)$.

\item Let $\mathbb{T}_n$ denote $(\mathbb{R}_{> 0})^n \cup \{(0,\ldots, 0)\}$.  We define $a\oplus b
= \mathrm{max}(a,b)$, where the maximum is taken with respect to the lexicographic ordering on the
positive orthant, and we define $a \otimes b$ to be component-wise multiplication.

\item The boolean semiring $\mathbb{B}$ is the sub-semiring $\{0,1\} \subset \mathbb{T}$.  Note that
$\mathbb{B}$ is the initial object in semirings.
\end{enumerate}
\end{example}

\subsection{Congruences and bend relations}
A quotient of a ring $R$ is defined by an equivalence relation on $R$ such that the ring structure
descends to the set of equivalence classes.  Such equivalence relations are of course in bijection
with ideals via the correspondence
\begin{align*}
	\text{ideal } I &\mapsto \text{equivalence relation } \{a \sim b \text{ if } a-b \in I\},\\
	\text{equivalence relation } K &\mapsto \text{ideal } \{a-b \:\: | \:\: a \sim_K b\}.
\end{align*}
This correspondence does not hold for semirings in general, and so we must work with the equivalence relations
themselves when defining quotients.

Recall that a \emph{congruence} on a semiring $S$ is an equivalence relation $K \subset S\times S$
that is also a subsemiring. If $K$ is a congruence on $S$, then the semiring structure on $S$
descends to a well-defined semiring structure on $S/K$.  Moreover, if $f\co S \twoheadrightarrow S'$ is a
surjective homomorphism of semirings then its kernel congruence $\mathrm{ker}\:f=  \{(a,b) \:\: | \:\: f(a) =
f(b)\}$ is indeed a congruence and $S/\mathrm{ker}\: f \cong S'$.

Given a set of binary relations $X \subset S\times S$, the congruence it generates can be described
concretely.  First take the subsemiring of $S\times S$ generated by $X$, and then take the
transitive and symmetric closure of this.  See \cite[Lemma 2.4.5]{GG1}.

We now come to a class of congruences on idempotent semirings that are essential in tropical
geometry. Given an expression $a_1 \oplus a_2 \oplus \cdots \oplus a_n$ in an idempotent semiring
$S$, the \emph{bend relations} of this expression, written $\bend(a_1 \oplus \cdots \oplus a_n)$, is
the congruence on $S$ generated by the relations
\[
	B_j\co \bigoplus_{i=1}^n a_i \sim \bigoplus_{i=1, i\neq j}^n a_i
\]
for each $j = 1 \ldots n$.  As special cases, when $n=2$, $\bend(a\oplus b)$ is
generated by the single relation $a \sim b$.  When $n=3$, $\bend(a\oplus b\oplus c)$ is generated by the relations
\[
	a \oplus b \oplus c \sim a \oplus b \sim a \oplus c \sim b\oplus c.
\]

The motivation for the bend relations stems from the following fact.
Recall that the \emph{tropical hypersurface} of a tropical polynomial $f \in \mathbb{T}[x_1, \ldots,
x_n]$ can be described away from the boundary of $\mathbb{T}^n$ as the locus where the graph of $f$
is non-linear (with respect to the $\mathbb{R} \cup \{\infty\}$ parametrization of $\mathbb{T}$).

\begin{proposition}[Prop. 5.1.6 of \cite{GG1}]
Given a tropical polynomial $f \in \mathbb{T}[x_1, \ldots, x_n]$, the tropical hypersurface of $f$
is precisely the set of homomorphisms $\mathbb{T}[x_1, \ldots, x_n] / \bend(f) \to \mathbb{T}$.
\end{proposition}

\subsection{Non-archimedean seminorms}\label{sec:seminorms}

A non-archimedean multiplicative seminorm on a ring $R$ is a map $v$ from  $R$  to the tropical
semiring $\mathbb{T}=(\mathbb{R}_{\geq 0}, \oplus = \mathrm{min}, \times)$ that is a homomorphism of
multiplicative monoids, has $v(0)=0$, and satisfies the ultrametric triangle inequality,
\[
	v(a+b) \leq v(a) \oplus v(b).
\]
The map $-\mathrm{log}\co \mathbb{R}_{\geq 0} \to \mathbb{R} \cup  \{\infty\}$  identifies
non-archimedean seminorms (always assumed to be multiplicative) with valuations, and so we will use the terms
interchangeably as we adopt the following generalization.

\begin{definition}\label{def:generalized-valuations}
A non-archimedean seminorm on a ring $R$ is an idempotent semiring $T$ and a map $v\co R \to T$ satisfying
\begin{enumerate}
\item	$v(0)  = 0_T$,
\item	$v(1) = v(-1) = 1_T$,
\item	$v(ab) = v(a) v(b)$,
\item	$v(a+b) \oplus v(a) \oplus v(b) = v(a) \oplus v(b)$.
\end{enumerate}
\end{definition}
\begin{remark}
Condition (4) generalizes the ultrametric triangle inequality, as it says that $v(a+b) \leq v(a) \oplus v(b)$ in
the canonical partial order on $T$.  This definition thus becomes equivalent to the usual definition
of a Krull valuation when the partial order is a total order, such as when $T =
\mathbb{T}_n$.  This condition can also be written more symmetrically as $v(a)\oplus v(b) \oplus
v(c) = v(a) \oplus v(b)$ whenever $a+b+c=0$ in $A$, since $v(c)=v(a+b)$.
\end{remark}

For use later on, we record the following simple observation.  A rank 1 valuation $v\co R \to
\mathbb{T}$ can be extended to a rank 2 valuation on the the ring $R\{\!\{t\}\!\}$ of Puiseux series
(or the subrings of formal Laurent series or polynomials) by the formula
\begin{equation}\label{eq:rank2valuation}
	a_0 t^{n_0} + \cdots \mapsto (e^{n_0}, v(a_0)) \in \mathbb{T}_2.
\end{equation}

\subsection{Tropical differential semirings}
Recall that a differential ring is a ring $R$ equipped with an additive map $d\co R \to R$ that
satisfies the Leibniz relations $d(ab) = (da)b + a(db)$.  This relation for 2-fold products easily
implies that $d$ also satisfies an analogous relation for $n$-fold products.  We will call these
relations the \emph{strict Leibniz relations}, and a map $d$ satisfying them will be called a
\emph{strict differential}.

In this work we propose that differentials on idempotent semirings should be required to satisfy a
somewhat weaker condition than the strict Leibniz relations.  Given an idempotent semiring $S$, an
additive map $d\co S \to S$ is said to be a \emph{tropical differential} if it satisfies the
\emph{tropical Leibniz relations}:  for any pairs of elements  $x,y \in S$ the bend
relations of the expression
\[
d(xy) \oplus xd(y) \oplus yd(x)
\]
hold. Note that we can view the tropical Leibniz relations as the tropicalization of the strict Leibniz relations.

\begin{definition}
A \emph{tropical differential semiring} is an idempotent semiring equipped with a tropical differential. A morphism of tropical differential semirings 
\[
f \co (S, d_S) \rightarrow (T, d_T)
\] 
is a morphism of semirings such that $f(d_S(x)) = d_T(f(x))$ for all $x \in S$.
\end{definition}

\begin{remark}
Just as the strict Leibniz relations for 2-fold products imply the strict Leibniz relations for $n$-fold products, such as
\[
	d(xyz) = (dx)yz + x(dy)z + xy(dz),
\]
it is possible to derive $n$-fold tropical Leibniz relations from the 2-fold tropical Leibniz relations.  I.e., in a tropical differential semiring the bend relations of any expression
\[
d(x_1 \cdots x_n) \oplus \bigoplus_i x_1 \cdots x_{i-1} (dx_i) x_{i+1} \cdots x_n
\]
hold.  However, the tropical Leibniz relations are distinct from the strict Leibniz relations in two important ways.
(1) A strict differential $d$ automatically satisfies the tropical Leibniz relations, but there are
many tropical differentials that are not strict. (2) The differential of a product $xy$ is
constrained by the tropical Leibniz relations and the differentials of $x$ and $y$, but it is not
uniquely determined by them.
\end{remark}

\begin{example}\label{ex:differential-idempotent-semirings}
\begin{enumerate}
\item Let $S$ be an idempotent semiring and let $d$ be either the zero map or the identity; these
are each strict differentials on $S$.

\item Consider the idempotent semiring $\mathbb{B}[[t]]$ of formal power series with coefficients in
$\mathbb{B}$.  This can be identified with the power set of $\mathbb{N}$ by sending a power series to
the set of exponents appearing in it; sum corresponds to union, and the product corresponds to
Minkowski sum.  The map defined by $t^n \mapsto t^{n-1}$ (for any $n \geq 1$ ) is a strict
differential.    If $p$ is a prime, the map defined by
\[
t^n \mapsto \begin{cases} 
t^{n-1} & n \geq 1 \text{ and } p \nmid n  \\
0 & n=0 \text{ or } p \mid n
\end{cases}
\]
is a tropical differential that is not strict.

\item Consider the idempotent semiring of formal tropical power series $\mathbb{T}[[t]]$.  It can be
endowed with a strict differential, $d_0$, defined by
\[
d_0(t^n) = \begin{cases} 
t^{n-1} & n \geq 1\\
0 & n=0.
\end{cases}
\]
\item More generally, if $v\co \mathbb{N} \to \mathbb{T}$ is a non-archimedean seminorm, then there is a non-strict tropical differential $d_v$
defined by,
\begin{equation}\label{eq:nontriv-d-dt}
d_v(t^n) = \begin{cases} 
v(n)t^{n-1} & n \geq 1\\
0 & n=0.
\end{cases}
\end{equation}
Indeed, this satisfies the tropical Leibniz relations since 
\[
	d_v(t^n t^m) \oplus d_v(t^n)t^m \oplus t^n d_v(t^m) = (v(n+m) \oplus v(n) \oplus v(m))t^{n+m-1},
\]
and the coefficient $v(n+m) \oplus v(n) \oplus v(m)$ on the right satisfies the bend relations (the
argument is essentially the same for $k$-fold products with $k>2$).  Note that $v$ could be either a
$p$-adic norm, or a degenerate $p$-adic norm where $v(n) = 0$ if $p$ divides $n$, and
$1$ otherwise.
\end{enumerate}
\end{example}

\begin{remark}
While $\mathbb{B}$ is the initial objects in the category of idempotent semirings, tropical differential semirings do not admit an initial object because the tropical Leibniz rule does not determine $d(1)$.  
\end{remark}

\subsection{Differential congruences}
Let $(S,d)$ be a semiring equipped with an additive map $d\co S \to S$.  A \emph{differential congruence} on $S$ is a congruence $K \subset S\times S$ that is closed under $d$; i.e., if $(a,b) \in K$ then $(da, db) \in K$.  When $K$ is a differential congruence, the map $d$ descends to an additive map $\overline{d}\co S/K \to S/K$, and if $d$ is a tropical differential then $\overline{d}$ is as well.

\begin{proposition}\label{prop:generating-a-differetial-congruence}
If $\{I_\lambda \subset S\times S\}$ is a set of differential congruences, then the congruence generated by them is a differential congruence.
\end{proposition}
\begin{proof}
Let $K$ denote the congruence generated by the $I_\lambda$.  It is the transitive and symmetric
closure of the subsemiring $K_0$ generated by the $I_\lambda$.

We will first show that $d(K_0) \subset K$.  Suppose that $(a_1,a_2)$ is a relation in some
$I_i$ and $(b_1, b_2)$ is a relation in some $I_j$.  Since $d$ is additive, it certainly sends the
sum $(a_1 + b_1, a_2 + b_2)$ to a relation in $K$.  For the product, we proceed as follows.  In
$I_i$ we have $(a_1 b_1, a_2 b_1)$ and hence $(d(a_1 b_1), d(a_2 b_1))$ since $I_i$ is a
differential congruence. Likewise, in $I_j$ we have the relations $(b_1 a_2, b_2 a_2)$ and hence
$(d(b_1 a_2), d(b_2 a_2))$.  Hence the relation $(d(a_1 b_1), d(a_2 b_2))$ is indeed contained in
the transitive closure $K$.  Since any element of $K_0$ is produced by a finite sequence of sums and
products of elements in the $I_\lambda$, it follows that $d(K_0) \subset K$, as desired.

Now, any relation in $K$ can be decomposed as a finite transitive chain of relations in $K_0$.  Thus
is follows that $d(K) \subset K$.
\end{proof}

\section{Differential polynomials}

The objective of this section is to construct a variant of the Ritt algebra of differential polynomials $R\{x_1,
\ldots, x_n\}$, where the coefficient ring $R$ is replaced by a tropical differential semiring $S$.  

\subsection{Classical Ritt algebras and their universal property}

Recall that, when $R$ is a differential ring, the Ritt algebra of differential polynomials $R\{x_1,
\ldots, x_n\}$ is the commutative $R$-algebra freely generated by the variables $x_i$ and their
formal derivatives $x_i^{(j)} = d^jx_i$.  It carries a differential that sends $x_i^{(j)}$ to
$x_i^{(j+1)}$ and extends as an additive derivation to arbitrary elements in $R\{x_1, \ldots,
x_n\}$.

Ritt algebras are characterized by a universal property: a homomorphism of differential rings
$\varphi\co R\{x_1, \ldots, x_n\} \to R'$ is uniquely determined by the images of the generators
$\varphi(x_i) \in R'$, and there are no restrictions on these elements. It is this universal
property that we wish to extend to the idempotent setting.

\subsection{The failure of the naive definition of tropical Ritt algebra}

While one can trivially replace the coefficient ring $R$ with a differential idempotent semiring
$S$ in the above Ritt algebra construction, we shall now see that it impossible to endow this with a
differential (either strict or tropical) for which the analogous universal property holds.  In fact,
we will show that there is no longer a unique choice of differential, and for any choice of
differential the universal property fails.

Suppose that $S$ is a differential idempotent semiring with tropical differential $d_S$.  The naive
definition of differential polynomials over $S$ suggested in the preceding paragraph will
nevertheless be useful later on, so we give it a name: 
\begin{definition}
Given a tropical differential semiring $S$, the algebra of \emph{basic differential polynomials} over $S$, denoted
\[
S\{x_1, \ldots, x_n\}_{basic}
\]
is the polynomial $S$-algebra on variables $x_i$ and their formal derivatives
$x_i^{(j)}$.
\end{definition}

Let us now attempt to extend the differential of $S$ to this algebra.  Obviously we would like to
send $x_i^{(j)}$ to $x_i^{(j+1)}$, and we would like the map to be additive.  The difficulty is in
choosing how to extend it to arbitrary products.  In contrast to the case of coefficients in
a differential ring, the tropical Leibniz relations allow more freedom in extending a
partially-defined differential to all products; there is not a uniquely determined extension of
$d_S$ to a map $d$ on all of $S\{x_1, \ldots, x_n\}_{\textit{basic}}$ satisfying the tropical Leibniz
relations. In fact, $S\{x_1, \ldots, x_n\}_{\textit{basic}}$ admits many distinct differentials.

\begin{example}\label{}
Suppose $w\co \mathbb{N} \to S$ is a non-archimedean seminorm.  Then we can define a differential $d_w$ on $S\{x_1,
\ldots, x_n\}_{\textit{basic}}$ by the following rule. First, for a pure power $(x^{(j)}_i)^k$ , we define
\[
d_w  ((x^{(j)}_i)^k )= w(k)(x^{(j)}_i)^{k-1}x^{(j+1)}_i
\]
Then, we extend this to monomials $c (x^{(j_1)}_{i_1})^{k_1} \cdots (x^{(j_m)}_{i_m})^{k_m}$ as a strict derivation.  E.g., 
\[
d_w(c x_1^a (x^{(3)}_2)^b) = d_S(c)x_1^a (x^{(3)}_2)^b \oplus cw(a)x_1^{a-1}x_1^{(1)} (x^{(3)}_2)^b \oplus cw(b)(x^{(3)}_2)^{b-1}x^{(4)}_2 x_1^a.
\]
It is straightforward to check that this map $d_w$ does indeed satisfy the tropical Leibniz relations.


One can generalize this example by choosing a distinct non-archimedean seminorm $w_{ij}$ for each generator $x_i^{(j)}$, defining the differential on pure powers by the rule
\[
d (( x_i^{(j)})^k) = w_{ij}(k) (x_i^{(j)})^{k-1} x_i^{(j+1)},
\]
and then extending to arbitrary monomials using the strict Leibniz rule.
\end{example}
The above example shows that there is at least one differential on $S\{x_1, \ldots, x_n\}_{basic}$
for each $n$-tuple of non-archimedean seminorms $\mathbb{Q} \to S$. 


\begin{proposition}
There is no tropical differential on $S\{x_1, \ldots, x_n\}_{\textit{basic}}$ that extends the
tropical differential on $S$ and makes this the free object on $n$ generators.
\end{proposition}
\begin{proof}
We use proof by contradiction.  Suppose $d$ is such a differential, so for any non-archimedean seminorm $w\co
\mathbb{N} \to S$ the identity map must be a morphism of differential idempotent semirings
\[
(S\{x_1, \ldots, x_n\}_{\textit{basic}},d) \to (S\{x_1, \ldots, x_n\}_{\textit{basic}}, d_w).
\]
This implies that $d=d_w$ on pure powers.  But this is a contradiction since $d_w \neq d_{w'}$ if
$w$ and $w'$ are distinct seminorms.
\end{proof}

We will show below that the idempotent semiring $S\{x_1, \ldots, x_n\}_{basic}$ can be enlarged to a
tropical differential semiring $S\{x_1, \ldots, x_n\}$ enjoying the universal property that justifies
calling it the tropical Ritt algebra.

\subsection{Differential polynomials over a ring}

As a warm-up to constructing tropical Ritt algebras, we first give an alternative construction of
the classical Ritt algebras in terms of trees.

To avoid ambiguity, let us be precise about our graph theory definitions and conventions.
\begin{definition}
A \emph{forest} is a finite set $V$ (the \emph{vertices}) together with a \emph{parent} map $P\co V
\to V$ such for $n$ large enough $P^n$ sends each vertex to a fixed point of $P$.  The fixed points
of $P$ are the \emph{roots}.  A \emph{tree} is a forest with a single root.  The \emph{valence} of a
vertex $v \in V$ is the cardinality of $P^{-1}(v) \smallsetminus \{v\}$, and the vertices of valence
0 are the \emph{leaves}.
\end{definition}

Consider the set of isomorphism classes of forests with leaves labelled by elements of $R
\cup \{x_1, \ldots, x_n\}$. This set becomes an abelian monoid under the operation of disjoint
union.  Now let $\mathrm{Forest}(R,n)$ denote the quotient monoid obtained by imposing the following
relations:
\begin{enumerate}
\item A tree with a leaf labelled by $0\in R$ is identified with the empty forest.

\item Any leaf labelled by $1$ can be deleted.

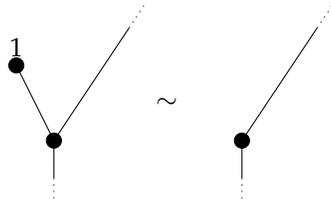
\begin{center}
\begin{tikzpicture}
\draw[fill] (0,0) circle [radius=0.1];
\draw (0,0) -- (1,1.5);
\draw [dotted] (1,1.5) -- (1.2,1.8);
\draw (0,0) -- (0,-0.5);
\draw [dotted] (0,-0.5) -- (0,-0.8);

\draw (0,0) -- (-0.5,1);
\draw[fill] (-0.5,1) circle [radius=0.1];
\node [above] at (-0.5,1) {$1$};

\node at (1.5,0.5) {$\sim$};

\draw[fill] (2.5,0) circle [radius=0.1];
\draw (2.5,0) -- (3.5,1.5);
\draw [dotted] (3.5,1.5) -- (3.7,1.8);
\draw (2.5,0) -- (2.5,-0.5);
\draw [dotted] (2.5,-0.5) -- (2.5,-0.8);

\end{tikzpicture}
\end{center}

\item If there are two leaves with labels $a, b \in R$ having the same parent in a tree, then we may
replace these two leaves with a single leaf labelled $ab$.
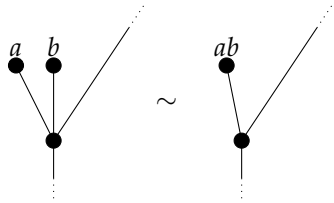
\begin{center}
\begin{tikzpicture}
\draw[fill] (0,0) circle [radius=0.1];
\draw (0,0) -- (1,1.5);
\draw [dotted] (1,1.5) -- (1.2,1.8);
\draw (0,0) -- (0,-0.5);
\draw [dotted] (0,-0.5) -- (0,-0.8);

\draw (0,0) -- (-0.5,1);
\draw[fill] (-0.5,1) circle [radius=0.1];
\node [above] at (-0.5,1) {$a$};

\draw (0,0) -- (0,1);
\draw[fill] (0,1) circle [radius=0.1];
\node [above] at (0,1) {$b$};

\node at (1.5,0.5) {$\sim$};

\draw[fill] (2.5,0) circle [radius=0.1];
\draw (2.5,0) -- (3.5,1.5);
\draw [dotted] (3.5,1.5) -- (3.7,1.8);
\draw (2.5,0) -- (2.5,-0.5);
\draw [dotted] (2.5,-0.5) -- (2.5,-0.8);

\draw (2.5,0) -- (2.3,1);
\draw[fill] (2.3,1) circle [radius=0.1];
\node [above] at (2.3,1) {$ab$};
\end{tikzpicture}
\end{center}

\item If a leaf with label $r\in R$ has a univalent non-root parent then we may replace the leaf and its
parent by a single leaf with label $dr$.

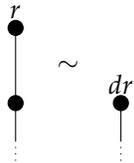
\begin{center}
\begin{tikzpicture}
\draw[fill] (0,0) circle [radius=0.1];
\draw (0,0) -- (0,-0.5);
\draw [dotted] (0,-0.5) -- (0,-0.8);

\draw (0,0) -- (0,1);
\draw[fill] (0,1) circle [radius=0.1];
\node [above] at (0,1) {$r$};

\node at (0.7,0.5) {$\sim$};

\draw[fill] (1.4,0) circle [radius=0.1];
\draw (1.4,0) -- (1.4,-0.5);
\draw [dotted] (1.4,-0.5) -- (1.4,-0.8);
\node [above] at (1.4,0) {$dr$};
\end{tikzpicture}
\end{center}

\item Given elements $r_1, r_2 \in R$ and a tree $t$, we may form trees $r_1 \cdot t, r_2\cdot t$,
	and $(r_1 + r_2)\cdot t$ by grafting a leaf with label $r_1, r_2$, and $(r_1 + r_2)$,
	respectively, at the root.  We then identify the pair of trees $(r_1 \cdot t) \cup (r_2 \cdot
	t)$ with the single tree $(r_1 + r_2)\cdot t$.
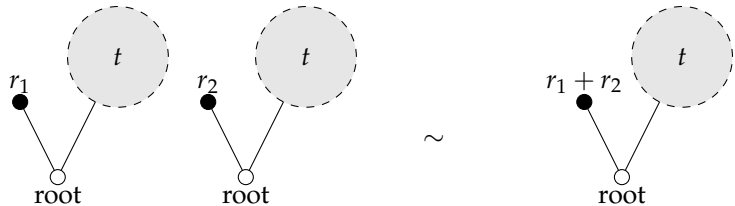
\begin{center}
\begin{tikzpicture}
\definecolor{palegray}{rgb}{0.9,0.9,0.9}
\draw (0,0) -- (0.5,1);
\draw[fill=palegray, dashed] (0.8,1.6) circle [radius=0.67];
\node at (0.8,1.6) {$t$};
\draw (0,0) -- (-0.5,1);
\draw[fill] (-0.5,1) circle [radius=0.1];
\node [above] at (-0.5,1) {$r_1$};
\draw[fill=white] (0,0) circle [radius=0.1];
\node [below] at (0,0) {root};

\draw (2.5,0) -- (3,1);
\draw[fill=palegray, dashed] (3.3,1.6) circle [radius=0.67];
\node at (3.3,1.6) {$t$};
\draw (2.5,0) -- (2,1);
\draw[fill] (2,1) circle [radius=0.1];
\node [above] at (2,1) {$r_2$};
\draw[fill=white] (2.5,0) circle [radius=0.1];
\node [below] at (2.5,0) {root};

\node at (5,0.5) {$\sim$};

\draw (7.5,0) -- (8,1);
\draw[fill=palegray, dashed] (8.3,1.6) circle [radius=0.67];
\node at (8.3,1.6) {$t$};
\draw (7.5,0) -- (7,1);
\draw[fill] (7,1) circle [radius=0.1];
\node [above] at (7,1) {$r_1 + r_2$};
\draw[fill=white] (7.5,0) circle [radius=0.1];
\node [below] at (7.5,0) {root};
\end{tikzpicture}
\end{center}

\end{enumerate}

We think of the set of leaves with a given parent as representing the monomial formed by multiplying
their labels, and we think of internal edges not incident at the root as representing the
differential $d$.  Thus a tree represents a differential monomial, i.e., an expression formed from
the elements of $R$ and the variable $x_1, \ldots, x_n$ by taking products and applying the
differential $d$.

\begin{example}
The expression $r_1 x_1 d^2(x_2)d(r_2 x_1 d(x_2))$ corresponds to the tree:
\begin{center}
\begin{tikzpicture}
\draw (0,0) -- (-1.4,1);
\draw[fill=black] (-1.4,1) circle [radius=0.1];
\node [above] at (-1.4,1) {$r_1$};

\draw (0,0) -- (-0.7,1);
\draw[fill=black] (-0.7,1) circle [radius=0.1];
\node [above] at (-0.7,1) {$x_1$};

\draw (0,0) -- (0,3);
\draw[fill=black] (0,1) circle [radius=0.1];
\draw[fill=black] (0,2) circle [radius=0.1];
\draw[fill=black] (0,3) circle [radius=0.1];
\node [above] at (0,3) {$x_2$};

\draw (0,0) -- (1,1);
\draw [fill=black] (1,1) circle [radius=0.1];
\draw (1,1) -- (0.8,2);
\draw [fill=black] (0.8,2) circle [radius=0.1];
\node [above] at (0.8,2) {$r_2$};
\draw (1,1) -- (1.6,2);
\draw [fill=black] (1.6,2) circle [radius=0.1];
\node [above] at (1.6,2) {$x_1$};
\draw (1,1) -- (2.4,2);
\draw [fill=black] (2.4,2) circle [radius=0.1];
\draw (2.4,2) -- (2.5,3);
\draw [fill=black] (2.5,3) circle [radius=0.1];
\node [above] at (2.5,3) {$x_2$};
\draw [fill=white] (0,0) circle [radius=0.1];

\end{tikzpicture}
\end{center}
\end{example}

\begin{proposition}
The monoid $\mathrm{Forest}(R,n)$ is a commutative $R$-algebra.
\end{proposition}
\begin{proof}
The addition operation is disjoint union (the original monoid operation).  The product of two trees
is defined by gluing their roots together, and we extend this operation to products of forests by
the distributive rule.   The elements  $r\in R$ sit inside $\mathrm{Forest}(R,n)$ as the trees
consisting of just a root and a single leaf with label $r$.  It is straightforward to verify that
this is an $R$-algebra.
\end{proof}

There is a map
\[
	d\co \mathrm{Forest}(R,n) \to \mathrm{Forest}(R,n)
\]
that inserts an edge at each root in a forest:
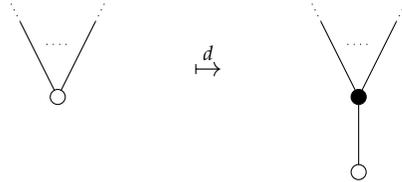
\begin{center}
\begin{tikzpicture}
\draw (0,0) -- (-0.5,1);
\draw [dotted] (-0.5,1) -- (-0.65,1.3);
\draw (0,0) -- (0.5,1);
\draw [dotted] (0.5,1) -- (0.65,1.3);
\draw [dotted] (-0.15,0.7) -- (0.15,0.7); 
\draw [fill=white] (0,0) circle [radius=0.1];

\node at (2,0.5) {$\stackrel{d}{\mapsto}$};

\draw (4,0) -- (3.5,1);
\draw [dotted] (3.5,1) -- (3.35,1.3);
\draw (4,0) -- (4.5,1);
\draw [dotted] (4.5,1) -- (4.65,1.3);
\draw [dotted] (3.85,0.7) -- (4.15,0.7); 
\draw [fill=black] (4,0) circle [radius=0.1];
\draw (4,0) -- (4,-1);
\draw [fill=white] (4,-1) circle [radius=0.1];

\end{tikzpicture}
\end{center}

This map is clearly not a derivation, so let 
\[
L \defeq \{d(st) - sdt - tds \:\: | \:\:  s,t \in \mathrm{Forest}(R,n) \}
\]
and let $\langle L \rangle$ denote the differential ideal generated by $L$, i.e., the smallest ideal
in $\mathrm{Forest}(R,n)$ that is closed under applying $d$.  By construction, the differential $d$
descends to a derivation on the quotient.

\begin{proposition}
There is a natural isomorphism of differential $R$-algebras,
\[
\mathrm{Forest}(R,n)/\langle L \rangle \cong R\{x_1, \ldots, x_n\}.
\]
\end{proposition}
\begin{proof}
It is straightforward to check that $\mathrm{Forest}(R,n)/\langle L \rangle$ has exactly the same universal property as the Ritt algebra.
\end{proof}

\subsection{Differential polynomials over a semiring}

We now give a variation on the above construction of the Ritt algebra via trees, where we replace
$R$ with a differential idempotent semiring $S$, and we replace the classical Leibniz relations with
the tropical Leibniz relations.

First consider the $S$-algebra $\mathrm{Forest}(S,n)$ defined exactly as above, and then consider the subset
\[
	L_{trop} = \{ d(t_1 t_2) \oplus t_2 dt_1 \oplus t_1dt_2 \:\: | \:\: t_1,t_2 \in \mathrm{Forest}(S,n)\}.
\]
Let $\langle L_{trop} \rangle$ denote the smallest ideal containing $L_{trop}$ and closed under applying $d$.

\begin{definition}
Given a differential idempotent semiring $S$, we define the \emph{tropical Ritt algebra}
\[
S\{x_1, \ldots, x_n\} \defeq \mathrm{Forest}(S,n) / \bend \langle L_{trop} \rangle,
\]
where $\bend \langle L_{trop} \rangle$ is the congruence of bend relations generated by $\langle L_{trop} \rangle$.
\end{definition}

The tropical Ritt algebra enjoys a universal property in the category of differential idempotent
semirings that is entirely analogous to the universal property of the classical Ritt algebra in the
category of differential rings.

\begin{proposition}\label{prop:univ-property-of-tropical-ritt}
Given a differential $S$-algebra $S'$, there is a bijection 
\[
	\mathrm{Hom}(S\{x_1, \ldots, x_n\}, S') \cong (S')^n
\]
implemented by sending a homomorphism $\varphi$ to the $n$-tuple $(\varphi(x_1), \ldots, \varphi(x_n))$.
\end{proposition}
\begin{proof}
By the construction of $\mathrm{Forest}(S,n)$, any $n$-tuple $(a_1, \ldots, a_n) \in (S')^n$ determines a homomorphism of semirings
\[
	\mathrm{Forest}(S,n) \to S'
\]
that commutes with $d$, and since the tropical Leibniz relations hold in $S'$, this homomorphism
descends to the quotient by $\bend \langle L_{trop} \rangle$.  Conversely, a homomorphism provides
an $n$-tuple of elements of $S'$.
\end{proof}

The algebra $S\{x_1, \ldots, x_n\}_{basic}$ of basic differential polynomials that was introduced
earlier sits inside $S\{x_1, \ldots, x_n\}$ as the set of forests where only the root vertices have
valence larger than 1; as the notation suggests.  Obviously the basic subalgebra is not closed under
taking differentials.

\section{Algebraic structures for tropical differential equations}

In the classical world, a differential equation over a differential ring $R$ is an element $f \in R\{x_1,
\ldots, x_n\}$, and a solution to $f$ in an $R$-algebra $A$ is an element $p \in A^n$ such that
$f(p)=0$.  Equivalently, $p$ is a solution if the corresponding homomorphism $p\co R\{x_1, \ldots,
x_n\} \to A$ factors through the quotient $R\{x_1, \ldots, x_n\}/(f)$.

In the tropical world, we have introduced differential idempotent semirings, but these objects on
their own are not sufficient to describe solutions to tropical differential equations.  A tropical
differential equation over a differential idempotent semiring $S$ is an element $f \in S\{x_1,
\ldots, x_n\}$ (where this is the tropical Ritt algebra defined above).  Solutions to this
differential equation will live in $S^n$, but asking that $f$ vanish or tropically vanish at $p\in
S^n$ turns out to be too restrictive.  Following the idea of Grigoriev's framework, $p$ should be
considered a solution to $f$ if $f$ tropically vanishes at $p$ \emph{to leading order} (rather than
to all orders).  This suggests that we must equip our differential idempotent semirings with
something like a non-archimedean seminorm that provides a way of measuring the leading order of elements.  To this
end, we will now define and study the category of tropical pairs.

\subsection{The category of tropical pairs}
A \emph{tropical pair} $\mathbf{S}$ consists of a tropical differential semiring $S_1$, an idempotent semiring $S_0$, and a
homomorphism of idempotent semirings $\pi\co S_1\to S_0$. 

\begin{remark}
We think of $S_1$ as a space of functions, and we think of $S_0$ as a space of leading exponents of the series expansions of these functions.  The map $\pi$, like the usual norm on Puiseux series, sends a function to its leading exponent.
\end{remark}

 In category theoretic terms, if 
\[
F\co\mathrm{DiffSemirings} \to \mathrm{Semirings}
\]
is the forgetful functor from differential idempotent semirings to idempotent semirings, then the
category of pairs is the simply the comma category $(F\downarrow \mathrm{Semirings})$.  Explicitly,
a morphism of pairs $\varphi$ from $(S_1\to S_0)$ to $(T_1
\to T_0)$ is a commutative diagram of idempotent semirings
\[
\begin{tikzcd}
S_1 \arrow[d] \arrow[r, "\varphi_1"] & T_1 \arrow[d] \\
S_0 \arrow[r, "\varphi_0"]           & T_0          
\end{tikzcd}
\]
in which the upper horizontal arrow $\varphi_1$ is a morphism of differential idempotent semirings.
Given a pair $\mathbf{S}$, the category of $\mathbf{S}$-algebras is the category of pairs $\mathbf{T}$ under $\mathbf{S}$.

A pair $(S_1 \stackrel{\pi}{\to} S_0)$ is said to be \emph{reduced} if $S_1$ admits no nontrivial quotient differential idempotent semiring over $S_0$; i.e., it is reduced if there is no nontrivial differential congruence contained in the congruence $\mathrm{ker}(\pi)$.
\begin{example}
\begin{enumerate}
\item For any morphism of idempotent semirings $\nu\co S \to T$ we have a pair 
\[
(S,d=0) \to T,
\] 
and it is reduced if and only if $\nu$ is injective.  If $\nu$ is not injective, then we
can replace $S$ with $\mathrm{im}(\nu)$ to obtain a reduced pair.

\item Consider the homomorphism \[\pi\co \mathbb{B}[[t]] \to \mathbb{T}\] defined by $t^n \mapsto
e^{-n}$, and endow $\mathbb{B}[[t]]$ with the differential $t^n \mapsto t^{n-1}$. This is a pair, and
it is reduced by the following argument.  Suppose $a \neq b \in \mathbb{B}[[t]]$.  If $a = \bigoplus
a_i t^i$ and $b = \bigoplus b_i t^i$, then there exists a minimal $n$ such that $a_n \neq b_n$.  It
then follows that $\pi(d^n a) \neq \pi(d^n b)$, and so $(d^n (a),d^n(b)) \notin
\mathrm{ker}(\pi)$.  Hence any differential congruence containing $(a,b)$ is not contained in
$\mathrm{ker}(\pi)$.


\item Consider 
\[ 
\pi\co \mathbb{T}[[t]]\to \mathbb{T}_2,
\] where the source has any of the differentials from Example
\ref{ex:differential-idempotent-semirings} and the morphism $\pi$ is given by
\[
	(a_{n_0} t^{n_0} \oplus a_{n_1} t^{n_1} \oplus \cdots) \mapsto (e^{-n_0}, a_{n_0}).
\] 
This is a pair, and a modification of the argument above shows that it is also reduced.
\end{enumerate}
\end{example}

We let $\reducedpairs$ denote the full subcategory of reduced pairs.  We will show below in Section
\ref{sec:reduction-functor} that $\reducedpairs$ is a reflective subcategory, and so any pair
$\mathbf{S}=(S_1 \to S_0)$ has a functorial reduction $\mathbf{S}^\mathit{red}=(S_1^{red} \to S_0)$.

Finally we are ready to define the category that will describe tropical differential equations and their solutions.
\begin{definition}
Given a reduced pair $\mathbf{S}$, an $\mathbf{S}$-algebra is a reduced pair under $\mathbf{S}$, and
we let $\mathbf{S}\mhyphen \mathrm{Alg}$ denote the category of $\mathbf{S}$-algebras.
\end{definition}

An important example of an $\mathbf{S}$-algebras comes from the tropical Ritt algebra.  Given a pair
$\mathbf{S}=(S_1\to S_0)$, we first define an idempotent semiring $(S_0 | S_1)\{x_1, \ldots, x_n\}$
by taking the pushout:
\begin{center}
\begin{tikzcd}
S_1 \arrow[r] \arrow[d] & S_1\{x_1, \ldots, x_n\} \arrow[d, dotted]\\
S_0 \arrow[r,dotted] & (S_0 | S_1)\{x_1, \ldots, x_n\}\\
\end{tikzcd}
\end{center}
This pushout can be described explicitly as the algebra of trees $S_1\{x_1, \ldots, x_n\}$  modulo
the congruence generated by the relations that identify $a,b\in S_1 \subset S_1\{x_1, \ldots, x_n\}$
if they have the same image in $S_0$; i.e. leaves incident at the root have labels in $S_0$ rather
than $S_1$. Note that $(S_0 | S_1)\{x_1, \ldots, x_n\}$ contains the
polynomial $S_0$-algebra $S_0\{x_1, \ldots, x_n\}_\mathit{basic}$.

The right vertical arrow in the above diagram gives an $\mathbf{S}$-algebra that we will
denote by $\mathbf{S}\{x_1, \ldots, x_n\}$; these pairs will play the role of tropical Ritt algebras
in the category of $\mathbf{S}$-algebras. 



\begin{proposition}\label{prop:ritt-pair-universal-property}
Let $\mathbf{S}=(S_1\to S_0)$ be a pair and $\mathbf{Y}=(Y_1 \to Y_0)$ an $\mathbf{S}$-algebra. Morphisms of $\mathbf{S}$-algebras $\varphi\co \mathbf{S}\{x_1, \ldots, x_n\} \to \mathbf{Y}$ are in bijection with $Y_1^n$.  The bijection is implemented
by sending a morphism $\varphi=(\varphi_1, \varphi_0)$ to $(\varphi_1(x_1),\ldots \varphi_1(x_n) )$.
\end{proposition}
\begin{proof}
Given $y\in Y_1^n$, it follows from the universal property of the tropical Ritt algebra (Prop.
\ref{prop:univ-property-of-tropical-ritt}) that there is a unique morphism of $S_1$-algebras
\[
	\varphi_1\co  S_1\{x_1,\ldots, x_n\} \to Y_1
\]
sending $x_i$ to $y_i$.  By the universal property of pushouts, this induces an arrow 
\[
\varphi_0\co (S_0|S_1)\{x_1,\ldots, x_n\} \to Y_0 
\]
 such that $(\varphi_1, \varphi_0)$ is a morphism of
$\mathbf{S}$-algebras, and this is unique since $\varphi_1$ is unique.
\end{proof}

\subsection{Tropical differential equations and their solutions}\label{sec:solution-sets}

We start with a reduced pair $\mathbf{S} = (S_1 \stackrel{\pi}{\to} S_0)$.  A tropical differential
equation is simply a differential polynomial $f \in S_0\{x_1, \ldots, x_n\}_\mathit{basic}$.  Let us
write $f = \sum_\alpha f_\alpha x^\alpha$, where $x^\alpha$ runs over the differential monomials in
$f$.  If $x^\alpha$ has any factors of the form $d^n x_i$ for $n>0$ then it does not make sense to
evaluate $x^\alpha$ at an element $c\in S_0^n$ because $S_0$ is not a differential semiring.
However, we can evaluate $x^\alpha$ at an element $C \in S_1^n$ and then push down to $S_0$ via
$\pi$.  Thus we can evaluate $f\in S_0\{x_1, \ldots, x_n\}_\mathit{basic}$ at $C\in S_1^n$ by the
expression
\[
	f(C) = \bigoplus_\alpha f_\alpha \pi(C^\alpha).	
\]

\begin{definition}
The solution set of a differential polynomial $f=\bigoplus_\alpha f_\alpha x^\alpha \in S_0\{x_1, \ldots,
x_n\}_\mathit{basic}$, denoted $\mathrm{Sol}(f)$, is the subset of $S_1^n$ consisting of all
elements $C= (C_1, \ldots, C_n)$ such that the sum
\[
	\bigoplus_\alpha f_\alpha \pi(C^\alpha)
\]
tropically vanishes.
\end{definition}

When the pair $S_1 \to S_0$ is $\mathbb{B}[[ t ]] \stackrel{\pi}{\to} \mathbb{T}$, the above
definition recovers Grigoriev's framework. A subset $W \subset \mathbb{N}$ corresponds to the
boolean formal power series $\bigoplus_{i \in W} t^i$, and Grigoriev's map $\mathrm{Val}_W(j)$ is
precisely $\pi(d^j W)$.

\begin{example}
Consider the pair $\mathbf{S} = \mathbb{T}[[t]] \stackrel{\pi}{\to} \mathbb{T}_2$, where $\mathbb{T}[[t]]$ has the
differential from Example \ref{ex:differential-idempotent-semirings} part (4) corresponding to the
2-adic norm, $d(t^n) = |n|_2t^{n-1}$.  Over this pair we consider solutions to the differential equation
\[
f=(e^{-4},1)x \oplus (1,8)x' \oplus (e^{-1}, 8)x'' \in \mathbb{T}_2\{x\}_\mathit{basic}.
\]	


Let us look for solutions of the form 
\[
	x = 1 \oplus \alpha t \oplus  \beta t^2 \oplus \gamma t^3 \oplus \delta t^4 \oplus \epsilon t^5 \oplus \cdots.
\]
We have 
\begin{align*}
x' &= \alpha  \oplus \frac{\beta}{2}t \oplus \gamma t^2 \oplus \frac{\delta}{4}t^3 \oplus \epsilon t^4 \oplus \cdots\\
x'' &= \frac{\beta }{2} \oplus \frac{\gamma}{2}t \oplus \frac{\delta}{4}t^2 \oplus \frac{\epsilon}{4}t^3 \oplus \cdots.
\end{align*}
If $\alpha \neq 0$ then
\[
	\pi(x) = (1,1), \quad \pi(x') = (1,\alpha), \quad \pi(x'') = (1,\beta/2),
\]
and so evaluating $f$ at $x$ gives the expression
\begin{align*}
	f(x) &= (e^{-4},1)(1,1) \oplus (1,8)(1,\alpha) \oplus (e^{-1}, 8)(1,\beta/2)\\
         &= (e^{-4},1) \oplus (1,8\alpha) \oplus (e^{-1},4\beta).
\end{align*}
The maximum occurs only in the middle term, so there is no solution with $\alpha\neq 0.$

Assuming next that $\alpha=0$ and $\beta \neq 0$, we have
\[
	\pi(x) = (1,1), \quad \pi(x') = (e^{-1},\beta/2), \quad \pi(x'') = (1,\beta/2),
\]
and
\begin{align*}
	f(x) &= (e^{-4},1)(1,1) \oplus (1,8)(e^{-1},\beta/2) \oplus (e^{-1}, 8)(1,\beta/2)\\
         &= (e^{-4},1) \oplus (e^{-1},4\beta) \oplus (e^{-1},4\beta).
\end{align*}
The second and third terms are equal and maximal, so this is a solution for any nonzero value of $\beta$.

If $\alpha=\beta=0$ and $\gamma\neq 0$, then
\[
	\pi(x) = (1,1), \quad \pi(x') = (e^{-2},\gamma), \quad \pi(x'') = (e^{-1},\gamma/2),
\]
and $f(x) = (e^{-4},1) \oplus (e^{-2},8\gamma) \oplus (e^{-2},4\gamma)$.
%
The middle term is the sole maximal term, so this is not a solution.

If $\alpha=\beta=\gamma=0$ and $\delta \neq 0$ then
\[
	\pi(x) = (1,1), \quad \pi(x') = (e^{-3},\delta/4), \quad \pi(x'') = (e^{-2},\delta/4),
\]
and $f(x) = (e^{-4},1) \oplus (e^{-2},2\delta) \oplus (e^{-2},2\delta),$
so we have a solution since the second and third terms are jointly maximal.

The last case we will look at is $\alpha=\beta=\gamma=\delta=0$ and $\epsilon\neq 0$.  Now
\[
	\pi(x) = (1,1), \quad \pi(x') = (e^{-4},\epsilon), \quad \pi(x'') = (e^{-3},\epsilon/4),
\]
and $f(x) = (e^{-4},1) \oplus (e^{-4},8\epsilon) \oplus (e^{-4},2\epsilon).$
If $\epsilon = 1/8$, then the first two terms are jointly maximal and we have a solution but when $\epsilon \neq 1/8$ either the first or second term is the sole maximum.  In this case we see for the first time that the tropical framework here provides additional information about solutions beyond the information contained in Grigoriev's framework.

\end{example}

\subsection{The reduction functor}\label{sec:reduction-functor}

Given a pair $\mathbf{S}= (S_1\stackrel{\pi}{\to} S_0)$, it follows from Proposition
\ref{prop:generating-a-differetial-congruence} that the set of differential congruences contained in
$\mathrm{ker} \: \pi$ has a unique maximal element $R(\pi)$, and hence the pair
$\mathbf{S}^{red} \defeq (S_1/R(\pi) \to S_0)$ is reduced.

\begin{proposition}\label{prop:reduction-universal-property}
The morphism $\mathbf{S} \to \mathbf{S}^{\mathit{red}}$ given by the projection onto the quotient $S_1 \twoheadrightarrow S_1/R(\pi)$ is initial among morphisms from $\mathbf{S}$ to reduced pairs.
\end{proposition}
\begin{proof}
Suppose $\mathbf{T}=(T_1 \stackrel{\psi}{\to} T_0)$ is a reduced pair and 
\[
\begin{tikzcd}
S_1 \arrow[d, "\pi"] \arrow[r, "\varphi_1"] & T_1 \arrow[d, "\psi"] \\
S_0 \arrow[r, "\varphi_0"]           & T_0          
\end{tikzcd}
\]
is a morphism of pairs.  There are inclusions
\[
 R(\pi) \subset \mathrm{ker} \: \pi \subset \mathrm{ker}\: \pi \circ \varphi_0.	
\]
The map $\varphi_1$ sends  $\mathrm{ker}\:  \varphi_0 \circ \pi$ into $\mathrm{ker}\: \psi$, and the image of a
differential congruence by a homomorphism of tropical differential semirings is again a differential
congruence, so $\varphi_1$ must send $R(\pi)$ to a differential congruence contained in
$\mathrm{ker}\: \psi$.  Since $\mathbf{T}$ is reduced, the only such differential congruence on $T_1$ is the
diagonal, and so $\varphi_1$ factors uniquely through the quotient map $S_1 \to S_1 / R(\pi)$.
\end{proof}

We now show that the above reduction construction exhibits $\reducedpairs$ as a reflective subcategory of $\pairs$.
\begin{proposition}\label{prop:reduction-is-left-adjoint}
Sending $\mathbf{S}$ to $\mathbf{S}^{\mathit{red}}$ defines a functor $\mathscr{R}\co \pairs \to \reducedpairs$, and the quotient
map $\mathbf{S} \to \mathbf{S}^\mathit{red}$ is a natural transformation $\mathrm{Id} \to \mathscr{R}$.  Moreover
$\mathscr{R}$ is left adjoint to the inclusion $\iota\co \pairs \hookrightarrow \reducedpairs$.   
\end{proposition}
\begin{proof}
Suppose $f\co\mathbf{S} \to \mathbf{T}$ is a morphism of pairs and consider the composition
$\mathbf{S} \to \mathbf{T} \to \mathbf{T}^\mathit{red}$.  By Proposition
\ref{prop:reduction-universal-property}, there is a unique factorization $\mathbf{S} \to
\mathbf{S}^\mathit{red} \to \mathbf{T}^\mathit{red}$, and hence we obtain a morphism
$\mathscr{R}(f)\co \mathbf{S}^\mathit{red} \to \mathbf{T}^\mathit{red}$.  It is straightforward the
check that this respects compositions:  $\mathscr{R}(f\circ g) = \mathscr{R}(f) \circ
\mathscr{R}(g)$.  Hence $\mathscr{R}$ is a functor.

It is a straightforward verification that the quotient map $\mathbf{S} \to \mathbf{S}^\mathit{red}$
defines a natural transformation from the identity on $\pairs$ to $\iota \circ \mathscr{R}$. Clearly
if $\mathbf{S}$ is reduced, then $\mathbf{S}^\mathit{red} = \mathbf{S}$, and there is trivially a
natural transformation from $\mathscr{R}\circ \iota$ to the identity on $\reducedpairs$.  It is now
elementary to check that these two natural transformations give the claimed adjunction.
\end{proof}

As a consequence of reduction being a left adjoint functor, it commutes with colimits.

\subsection{Quotients of pairs}\label{sec:quotients}
Let $\mathbf{S}=(S_1 \to S_0)$ be a pair.  A quotient of $\mathbf{S}$ is a morphism of pairs
\[
\begin{tikzcd}
S_1 \arrow[d,"\pi"] \arrow[r, "f_1"] & T_1 \arrow[d] \\
S_0 \arrow[r, "f_0"]           & T_0          
\end{tikzcd}
\]
such that both $f_1$ and $f_0$ are surjective.  The kernel of $f_1$ is a differential congruence
$\ker f_1$ on $S_1$, the kernel of $f_0$ is a congruence $\ker f_0$ on $S_0$, and $\pi$ sends $\ker
f_1$ into $\ker f_0$.  Conversely, a pair of congruences $(K_1 \subset S_1 \times S_1, K_0 \subset
S_0 \times S_0)$ satisfying $\pi(K_1) \subset K_0$ defines a quotient of $\mathbf{S}$.

We now describe an important class of quotients.  Suppose we are given a pair $\mathbf{S}=(S_1 \to
S_0)$ and a congruence $K$ on the polynomial semiring $S_0\{x_1, \ldots, x_n\}_{basic}$.  By a
slight abuse of notation, let $(S_0|S_1)\{x_1, \ldots, x_n\}/K$ denote the induced quotient, and then let
\[
	\mathbf{S}\{x_1, \ldots, x_n\} \sslash K 
\]
denote the reduction of the pair $S_1\{x_1, \ldots, x_n\} \to (S_0|S_1)\{x_1, \ldots, x_n\}/K$. Quotients
of this form will be used when we define the tropicalization of a system of differential equations
in Section \ref{sec:tropicalization}.

\begin{proposition}\label{prop:quotient-property}
Let $\mathbf{T} = (T_1 \stackrel{\pi}{\to} T_0)$ be a reduced $\mathbf{S}$-algebra and $K$ a congruence on
$S_0\{x_1, \ldots, x_n\}_\mathit{basic}$.  Morphisms of $\mathbf{S}$-algebras
\[
	\mathbf{S}\{x_1, \ldots, x_n\} \sslash K  \to \mathbf{T}
\]
correspond bijectively with $n$-tuples $y_1, \ldots, y_n \in T_1$ such that the elements $\pi(d^jy_i) \in T_0$ define an
$S_0$-algebra homomorphism $S_0\{x_1, \ldots, x_n\}_\mathit{basic} / K \to T_0$.
\end{proposition}
\begin{proof}
A morphism $(f_1, f_0)\co \mathbf{S}\{x_1, \ldots, x_n\} \sslash K  \to \mathbf{T}$ determines
elements $y_i = f_1(x_i)$ that are immediately seen to satisfy the stated condition.

Going in the other direction, suppose $y_i \in T_1$ are elements satisfying the above condition. By Proposition
\ref{prop:ritt-pair-universal-property},  there is a uniquely determined morphism  $(f_1, f_0):
\mathbf{S}\{x_1, \ldots, x_n\}  \to \mathbf{T}$ with $f_1(x_i) = y_i$. The images $\pi(d^jy_i)\in
T_0$ are equal to the elements $f_0(d^jx_i)$ coming from $S_0\{x_1, \ldots, x_n\}_\mathit{basic}$
(recall that this is the polynomial algebra on the symbols $d^jx_i$), and these define a semiring
homomorphism
\[
S_0\{x_1, \ldots, x_n\}_\mathit{basic}  \to T_0.	
\]
that descends to the quotient by $K$, and since $\mathbf{T}$ is reduced, this descends to the
reduction by Proposition \ref{prop:reduction-is-left-adjoint}.
\end{proof}

\subsection{Solutions as morphisms}

It follows directly from Proposition \ref{prop:quotient-property} and the definition of solutions to tropical differential equations that $\mathbf{S}$-algebra morphisms
\[
	\mathbf{S}\{x_1, \ldots, x_n\} \sslash \bend(E) \to \mathbf{S}
\]
are in bijection with the solution set $\mathrm{Sol}(E)$. In fact, we have
\begin{proposition}
The functor $\mathbf{S}\mhyphen \mathrm{alg} \to \mathrm{Sets}$ sending an $\mathbf{S}$-algebra $\mathbf{S} \stackrel{u}{\to} \mathbf{T}$ to $\mathrm{Sol}(u_* E)$ is corepresented by  $\mathbf{S}\{x_1, \ldots, x_n\} \sslash \bend(E)$.
\end{proposition}

\subsection{Colimits of pairs}
In this section we show that colimits in the category of pairs can be computed by
computing the colimits of the top and bottom individually.  In order for this to be
useful, it is helpful to note the following.

\begin{proposition}
The categories of idempotent semirings and differential idempotent semirings are cocomplete.
\end{proposition}
\begin{proof}
The category of idempotent semirings is cocomplete for the same reason as the category of rings; one
can easily check that arbitrary coproducts and coequalizers exist.  For differential idempotent
semirings, one must only verify that tropical differentials $d_i$ on $S_i$ induce a tropical
differential on the coproduct $\bigoplus_i S_i$, and likewise for coequalizers.  Both of these
verifications are elementary and straightforward.
\end{proof}

\begin{proposition}\label{prop:colim-components}
The forgetful functors
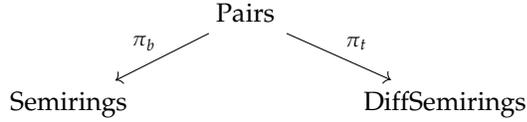
\begin{center}
\begin{tikzcd}
 & \pairs \arrow[dr, "\pi_t"] \arrow[dl, "\pi_{b}"']& \\
 \semirings & & \Dsemirings
\end{tikzcd}
\end{center}
commute with colimits, and $\pi_t$ also commutes with limits.
\end{proposition}

\begin{proof}
It suffices to show that $\pi_b$ admits a right adjoint and $\pi_t$ admits both a left and a right adjoint. 

We start with $\pi_t$.  Let 
\[
L_t\co \Dsemirings \to \pairs
\]
be the functor sending a differential idempotent
semiring $S$ to the pair $S\stackrel{id}{\to} S$, and let $R_t$ be the functor sending $S$ to the
pair $S\to *$, where $*$ denotes the trivial idempotent semiring consisting of a single element. Given a pair
$A \stackrel{p}{\to} B$, a morphism of differential idempotent semirings $f\co S \to A$ uniquely determines, and
is uniquely determined by, a morphism of pairs
\begin{center}
\begin{tikzcd}
S \arrow[d,"id"'] \arrow[r,"f"] & A \arrow[d,"p"] \\
S \arrow[r,"p\circ f"] & B
\end{tikzcd}
\end{center}
that is evidently natural in the semiring $S$ and the pair $A \to B$. Thus $L_t$ is left adjoint to $\pi_t$. 
For $R_t$, observe that a morphism of differential semirings $f\co A\to S$ is equivalent to a morphism of pairs:
\begin{center}
\begin{tikzcd}
A \arrow[d,"p"'] \arrow[r,"f"] & S \arrow[d] \\
B \arrow[r] & *
\end{tikzcd}
\end{center}

For $\pi_b$, we will construct a right adjoint $R_b$.  Consider the subcategory
\[\pairs/T \subset \pairs\] of pairs $S \to T$, where a morphism is a morphism of pairs that is the identity on $T$.
The colimit $\colim_{\pairs / T} \pi_t$ comes with a natural semiring homomorphism to $T$, and this defines a pair $R_b(T)$.  It is straightforward to verify that $R_b(T)$ is functorial in $T$.
A morphism of pairs 
\begin{center}
\begin{tikzcd}
A \arrow[d,"p"'] \arrow[r,"f"] & \colim_{\pairs / T} \pi_t \arrow[d] \\
B \arrow[r] & T
\end{tikzcd}
\end{center}
clearly provides a semiring homomorphism $B \to T$.  Conversely, given a semiring homomorphism $B \to T$, the composition $A\to B \to T$ is an object of $\pairs/T$ and hence it has a canonical map to $R_b(T)$. 
\end{proof}

Finally, note that since the reduction functor is idempotent and has a left adjoint (Proposition
 \ref{prop:reduction-is-left-adjoint}), the colimit of a diagram of reduced pairs is reduced.

\subsection{Differential enhancements of seminorms}\label{sec:differential-enhancement}

Given a seminorm $v$ on $R$ as in Definition \ref{def:generalized-valuations}, $v(x)$ does not
in general determine the seminorm of derivatives of $x$.  In order to define tropical differential equations,
we must enhance the seminorm with some additional information in order to determine the seminorms of
sequences $a, da, d^2a, \ldots$.  To this end, we now introduce differential enhancements of seminorms.
\begin{definition}
Given a differential ring $R$ and a non-archimedean seminorm $v\co R \to S_0$, a \emph{differential enhancement of $v$} is a reduced pair $\mathbf{S} = (S_1\to S_0)$ and a map of sets $\widetilde{v}\co R \to S_1$ such that
\begin{enumerate}
\item $\widetilde{v}(0) = 0\in S_1$ and $\widetilde{v}(1) = 1\in S_1$;
\item it commutes with the differentials: $d_{S_1} \widetilde{v}(x) = \widetilde{v}(d_R x)$ for any $x\in R$;
\item the following diagram commutes:
\begin{center}
\begin{tikzcd}
                        & S_1 \arrow[d] \\
R \arrow[ur, "\widetilde{v}"] \arrow[r,"v"] & S_0 \\
\end{tikzcd}
\end{center}
\end{enumerate}

We will use the term \emph{differentially enhanced seminorm} $\mathbf{v}=(v,\widetilde{v})\co A \to \mathbf{S}$ to mean a seminorm $v$ together with a differential enhancement $\widetilde{v}$.  
\end{definition}

Note that if $(v,\widetilde{v})\co A \to \mathbf{S}$ is a differentially enhanced seminorm and  $(f_0, f_1)\co \mathbf{S} \to  \mathbf{T}$ is a morphism of pairs, then the composition 
\begin{center}
\begin{tikzcd}
       & S_1 \arrow[d] \arrow[r, "f_1"] & T_1 \arrow[d] \\
 R   \arrow[ur,"\widetilde{v}"] \arrow[r,"v"] & S_0 \arrow[r,"f_0"]   & T_0     
\end{tikzcd}
\end{center}
is also a differentially enhanced seminorm.

\begin{example}\label{ex:grigoriev-seminorm}
Let $k$ be a field and consider the differential ring of formal power series $k[[t]]$ with differential $d/dt$.  The $t$-adic norm $k[[t]] \to \mathbb{T}$  admits a differential enhancement 
\begin{center}
\begin{tikzcd}
                        & \mathbb{B}[[t]] \arrow[d] \\
k[[t]] \arrow[ur, "\widetilde{v}"] \arrow[r,"v"] & \mathbb{T} \\
\end{tikzcd}
\end{center}
in which the map $\mathbb{B}[[t]] \to \mathbb{T}$ sends a boolean power series $t^n \oplus \cdots$ to $\mathrm{exp}(-n)$.  Note that while $v$ is multiplicative, its differential enhancement $\widetilde{v}$ is not.  For instance, 
\[	
\widetilde{v}( (1+t)(1-t) ) = 1 \oplus t^2,  \text{ whereas } \widetilde{v}(1+t) \cdot
\widetilde{v}(1-t) = 1 \oplus t \oplus t^2.
\]
This is the differentially enhanced seminorm used by Grigoriev \cite{Grigoriev-TDE} in his framework
and subsequent works
\cite{AGT-fundamental-theorem,Cotterill-DE,Falkensteiner-Fundamental-theorem,Fink-differential-initial-forms}.
\end{example}

\begin{example}
Consider the $p$-adic seminorm $v_p\co \mathbb{Q} \to \mathbb{T}$ and extend this to a seminorm $\mathbb{Q}[[t]] \to \mathbb{T}_2$ as in \eqref{eq:rank2valuation}.  This admits a differential enhancement
\begin{center}
\begin{tikzcd}
                        & \mathbb{T}[[t]] \arrow[d] \\
Q[[t]] \arrow[ur, "\widetilde{u}"] \arrow[r,"u"] & \mathbb{T}_2 \\
\end{tikzcd}
\end{center}
where the differential on $\mathbb{T}[[t]]$ is by \eqref{eq:nontriv-d-dt}, and the vertical arrow sends $a_0 t^{n_0} \oplus \cdots$ to $(\mathrm{exp}(-n_0), a_0)$.  Let $\mathbf{u} =(u,\widetilde{u})$.
There is a morphism of pairs
\begin{center}
\begin{tikzcd}
\mathbb{T}[[t]] \arrow[d] \arrow[r] & \mathbb{B}[[t]] \arrow[d]\\
\mathbb{T}_2 \arrow[r] & \mathbb{T}
\end{tikzcd}
\end{center}
given on the top by the sending all non-zero coefficients to 1, and on the bottom by projection onto the first component.
This morphism of pairs sends the differentially enhanced seminorm $\mathbf{u}$ to the $\mathbf{v}$ of Example \ref{ex:grigoriev-seminorm}. Thus $\mathbf{u}$ provides a refinement of the structure considered by Grigoriev.
\end{example}

While a seminorm $u\co R \to S_0$ may admit multiple distinct differential enhancements $\widetilde{u}$, as illustrated in the examples above, it turns out that there is at most one for any given reduced pair $\mathbf{S}$ over $S_0$.

\begin{proposition} \label{prop:unique-diff-enh}
Let $u\co R \to S_0$ be a seminorm and $S_1 \stackrel{\pi}{\to} S_0$ a reduced pair.  If $\widetilde{u},\widetilde{u}'\co R \to S_1$ are two differential enhancements of $u$, then $\widetilde{u} = \widetilde{u}'$.
\end{proposition}
\begin{proof}
Consider the congruence $K$ on $S_1$ generated by the relations $\widetilde{u}(x) \sim \widetilde{u}'(x)$ for $x\in R$.  Since $\widetilde{u}$ and $\widetilde{u}'$ both commute with the differentials, $K$ is a differential congruence, and since $\pi\circ \widetilde{u} = \pi \circ \widetilde{u}'$, it follows that $K \subset \ker \pi$.  Now, since $S_1 \to S_0$ is reduced, $K$ must be trivial.
\end{proof}

In a differential ring $R$, an element $a\in R$ is said to be a \emph{constant} if $d(a)=0$.  The
constants form a subring of $R$.
\begin{proposition}
Given a differentially enhanced seminorm $\mathbf{v}=(\widetilde{v}\co R \to S_1, v\co R \to S_0)$ on $R$, $\widetilde{v}$ restricts to a seminorm on the subring of constants in $R$.
\end{proposition}\label{prop:val-on-constants}
\begin{proof}

Suppose $a,b$ are constants and consider the semiring congruence $K$ on $S_1$ generated by the relations 
\begin{align*}
\widetilde{v}(0) & \sim 0_{S_1}\\
\widetilde{v}(1)  & \sim \widetilde{v}(-1) \\ 
\widetilde{v}(ab)  & \sim \widetilde{v}(a)\widetilde{v}(b)\\
\widetilde{v}(a + b) \oplus	\widetilde{v}(a) \oplus \widetilde{v}(b) & \sim \widetilde{v}(a) \oplus \widetilde{v}(b).
\end{align*}
Since $v$ is a non-archimedean seminorm, the relations $v(a+b) \oplus v(a) \oplus v(b) = v(a) \oplus v(b)$ and
$v(ab) = v(a)v(b)$ hold in $S_0$, and hence the semiring homomorphism $S_1 \to S_0$ factors through
the quotient semiring $S_1/K$ because each of the generators of $K$ (as a semiring congruence) is a
relation that holds in $S_0$.     Since $\widetilde{v}$ commutes with the differentials,
$\widetilde{v}(1)$, $\widetilde{v}(-1)$, $\widetilde{v}(a)$, $\widetilde{v}(b)$,
$\widetilde{v}(a+b)$ and $\widetilde{v}(ab)$ are each constants in $S_1$.  From this we see that $K$
is in fact a congruence of differential semirings.  If $K$ were nontrivial then the factorization
$S_1\to S_1/K \to S_0$ would contradict the fact that $S_1\to S_0$ is reduced.  Thus the equalities
\begin{align*}
\widetilde{v}(0) & = 0_{S_1}\\
\widetilde{v}(1)  & = v(-1)\\ 
	\widetilde{v}(a)\widetilde{v}(b) &= \widetilde{v}(ab),\\
 	\widetilde{v}(-a-b) \oplus  \widetilde{v}(a) \oplus \widetilde{v}(b) &= \widetilde{v}(a) \oplus \widetilde{v}(b)
\end{align*}
must hold in $S_1$.
\end{proof}

\subsection{The differential Berkovich space}

Let $k$ be a ring with a non-archimedean seminorm $v\co k \to \mathbb{T}$.  Given a $k$-algebra $A$, recall that the
Berkovich analytification of $\spec A$ is the set of seminorms $w\co A \to \mathbb{T}$ that are
compatible with $v$ in the sense that the composition $k \to A \stackrel{w}{\to} \mathbb{T}$ is
equal to $v$.  The analytification is denoted $(\spec A)^{\mathit{an}}$.  It is equipped with a
topology that we will not discuss here.

We now propose a generalization to the differential setting. Suppose $k$ is a differential ring
equipped with a differentially enhanced seminorm $\mathbb{v}$ to $\mathbf{S}= (S_1 \to S_0)$, and let $A$ be a
differential $k$-algebra.  Given an $\mathbf {S}$-algebra $\mathbf{T}$, a differentially enhanced seminorm
$\mathbf{w}=(\widetilde{w}, w)$ is said to be compatible with $\mathbf{v}$ if the diagram
\begin{center}
\begin{tikzcd}
          & S_1 \arrow[r] \arrow[dd] & T_1 \arrow[dd] \\
A \arrow[rrd, crossing over, near end, "w"] \arrow[rd, "v"'] \arrow[ru, "\widetilde{v}"] \arrow[rru, crossing over, "\widetilde{w}"', near end] 
          &                          &                \\
          & S_0 \arrow[r]            & T_0           
\end{tikzcd}
\end{center}
commutes.

\begin{definition}
Given an $\mathbf{S}$-algebra $\mathbf{T}$, the \emph{differential Berkovich space} of $A$, denoted
$\mathit{Berk}_{\mathbf{T}}(A)$, is the set of differentially enhanced seminorms $\mathbf{w}\co A \to
\mathbf{T}$ that are compatible with $\mathbf{v}$.
\end{definition}

Note that there is a natural map $\mathit{Berk}_{\mathbf{T}}(A) \to (\spec A)^\mathit{an}$ induced
by sending a differentially enhanced seminorm $\mathbf{w}=(\widetilde{w},w)$ to its underlying ordinary seminorm $w$. This map is injective thanks to Proposition \ref{prop:unique-diff-enh}.

\section{Tropicalization}\label{sec:tropicalization}

\subsection{Review of the non-differential case}
In the familiar non-differential setting, one starts with a field $k$ with a non-archimedean seminorm $v\co k \to
\mathbb{T}$, and one then defines three tropicalization maps:
\begin{enumerate}
\item Tropicalization of points is the map $\trop\co k^n \to \mathbb{T}^n$ given by applying the seminorm coordinate-wise.
\item Tropicalization of equations is the map $\trop\co k[x_1,\ldots, x_n] \to \mathbb{T}[x_1, \ldots, x_n]$
given by applying the seminorm coefficient-wise.  This extends to a map sending ideals in
$k[x_1,\ldots, x_n]$ to ideals in $\mathbb{T}[x_1,\ldots, x_n]$.
\item Tropicalization of varieties sends $V(I)$ to the subset of $\mathbb{T}^n$ defined by the intersection of the tropical
hypersurfaces of all $f \in \trop(I)$
\end{enumerate}
An ideal $I\subset k[x_1, \ldots, x_n]$ is of course a system of polynomial equations, and a
solution to this system is the same as a homomorphism $k[x_1,\ldots, x_n] /I \to k$ (or a
homomorphism to some $k$-algebra $A$). Since the tropical hypersurface of a tropical polynomial $f$
is exactly the solution set of the bend relations of $\trop(f)$, it follows that the tropicalization
of a variety is the set of solutions to the bend relations of the tropicalization of its defining
ideal.  Moreover, solutions to these bend relations are precisely homomorphisms of semirings
$\mathbb{T}[x_1, \ldots, x_n]/\bend \trop(I) \to \mathbb{T}$.  One can thus think of the semiring
$\mathbb{T}[x_1, \ldots, x_n]/\bend \trop(I)$ as the coordinate algebra of the tropical variety, and
hence tropicalization of varieties has an incarnation at the level of algebras given by
\[
	k[x_1,\ldots, x_n] /I \mapsto \mathbb{T}[x_1,\ldots, x_n] /\bend \trop(I)
\]
Note that this is a construction carried out when the seminorm $v$ takes values in an
idempotent semiring, not just $\mathbb{T}$; see \cite{GG1} for further details.


\subsection{Differential tropicalization}
We now turn to the differential setting.  Let $k$ be a differential ring equipped with a
differentially enhanced seminorm $\mathbf{v}=(\widetilde{v},v)\co k \to \mathbf{S}=(S_1 \stackrel{\pi}{\to} S_0)$
\begin{enumerate}
\item We tropicalize points $p\in k^n$ via the map $\trop\co k^n \to S_1^n$ defined by applying
$\widetilde{v}$ component-wise.
\item We tropicalize differential equations by applying $v$ coefficient-wise to define a map 
\[
\trop\co k\{x_1, \ldots, x_n\} \to S_0\{x_1, \ldots, x_n\}_\mathit{basic}.
\]
We write $\trop(I)$ for the ideal generated by the image of $I$, and so there is an induced a map
sending ideals in  $k\{x_1,
\ldots, x_n\}$ to ideals in $S_0\{x_1, \ldots, x_n\}_\mathit{basic}$. 
\item We use the tropicalization of equations map to define a construction sending quotients
$\alpha\co k\{x_1, \ldots, x_n\} \twoheadrightarrow k\{x_1, \ldots, x_n\}/I$ to quotients
$\mathbf{trop}(\alpha)$ of the pair $\mathbf{S}\{x_1, \ldots, x_n\}$.  Define
\[
\mathbf{trop}(\alpha) \defeq \mathbf{S}\{x_1, \ldots, x_n\} \sslash \bend \trop(I),
\] 
where $\bend \trop(I)$ is the congruence on $(S_0 | S_1)\{x_1, \ldots, x_n\}$ generated by the bend
relations of $\trop(I)$ and we use the quotient construction at the end of Section
\ref{sec:quotients}. 
\end{enumerate}

\begin{remark}
Since $(S_0 | S_1)\{x_1, \ldots, x_n\}$ is not a polynomial algebra, we cannot form the bend
relations of an arbitrary element in it.  The above construction uses the fact that applying $v$
coefficient-wise lands in $S_0\{x_1, \ldots, x_n\}_\mathit{basic} \subset (S_0 | S_1)\{x_1, \ldots,
x_n\}$, and this algebra is a polynomial algebra so we can form bend relations in it.
\end{remark}

\begin{proposition}
Given a differential ideal $I \subset k\{x_1, \ldots, x_n\}$, tropicalization of points 
\[
\trop\co k^n \to S_1^n\]
 sends $\mathrm{Sol}(I)$ into $\mathrm{Sol}(\trop(I))$.
\end{proposition}
\begin{proof}
It suffices to show that if $p\in k^n$ is a solution to $f \in k\{x_1, \ldots, x_n\}$, then
$\trop(p)\in S_1^n$ is a solution to $\trop(f) \in S_0\{x_1, \ldots, x_n\}_\mathit{basic}$.  Write $f=
\sum_{\epsilon \in \mathrm{supp} f} f_\epsilon x^\epsilon$, where $x^\epsilon$ denotes a
differential monomial in the variables $x_i$.   We have $f(p) = 0$, so $v(\sum_\epsilon f_\epsilon
p^\epsilon) = 0$ in $S_0$.  Since $v\co k \to S_0$ is a non-archimedean seminorm, this happens if and only if the sum
\[
	\sum_\epsilon v(f_\epsilon p^\epsilon),
\]
tropically vanishes.  Since the differential enhancement map $\widetilde{v}$ commutes with the
differentials, we have that $v(f_\epsilon p^\epsilon)$  is equal to $v(f_\epsilon)
\pi(\widetilde{v}(p)^\epsilon)$, which is equal to the evaluation of the differential monomial
$v(f_\epsilon)x^\epsilon$ at the point $\trop(p)$.  Thus $\trop(f)$ tropically vanishes at
$\trop(p)$.
\end{proof}

\subsection{Functoriality of tropicalization}

Tropicalization of differential equations sends a presentation of a differential algebra to a
tropical pair.  Here we show that this defines a functor from a category of presentations to the
category of tropical pairs.

A homomorphism of differential algebras \[f\co k\{x_1, \ldots,  x_n\} \to k\{y_1, \ldots,  y_m\}\] is
said to be \emph{monomial} if each variable $x_i$ is sent to a monomial (with coefficient 1) in the
variables $y_i$ and their derivatives (it need not send monomials to monomials in general), and it
is said to be \emph{linearly monomial} if each $x_i$ is sent to some $y_j$, a derivative of $y_j$, or to 0.

\begin{example}
The map $f\co k\{x\} \to k\{y\}$ given by $f(x) = y dy$ is monomial even though it sends the monomial
$dx$ to $(dy)^2 + yd^2y$.  However, it is not linearly monomial.
\end{example}

We now define the category of presentations $\mathrm{Pres}(A)$. Objects of this category are
presentations of $A$; i.e., an object is a Ritt algebra $k\{x_i \: | i \in \Lambda\}$ together with
a surjective homomorphism $k\{x_i \: | i \in \Lambda\} \to A$, and whose morphisms are commutative
triangles
\begin{equation}\label{eq:monomial-morphism}
\begin{tikzcd}
k\{x_1, \ldots, x_n\} \arrow[rr,"f"] \arrow[dr] &      & k\{y_1, \ldots, y_m\} \arrow[dl] \\
                                      &  A  & 
\end{tikzcd}
\end{equation}
where $f$ is a monomial morphism.  We allow the set $\Lambda$ of variables to be infinite, and let
$\mathrm{Pres}^\mathit{fin}(A)$ denote the subcategory of finite presentations.   Let
$\mathrm{Pres}_\mathit{lin}(A) \subset \mathrm{Pres}(A)$ denote the subcategory of presentations and
linearly monomial morphisms.


\begin{proposition}
The tropicalization construction $(k\{x_i \: | \: i \in \Lambda\} \stackrel{\alpha}{\to} A) \mapsto \mathbf{trop}(\alpha)$ yields a functor $\mathrm{Pres}(A) \to \mathbf{S}\mathrm{\mhyphen alg}$.
\end{proposition}
\begin{proof}
Given a morphism of presentations
\[
\begin{tikzcd}
k\{x_i \: | \: i \in \Lambda_1\} \arrow[rr,"f"] \arrow[dr, "\alpha"'] &      & k\{y_j \: | \: j \in \Lambda_2\} \arrow[dl,"\beta"] \\
                                      &  A  & 
\end{tikzcd}
\]
each monomial in $k\{x_i \: | \: i \in \Lambda_1\}$ corresponds to a monomial in the tropical Ritt
algebra $\mathbf{S}\{x_i \: | \: i \in \Lambda_1\}$, and so it follows from Proposition
\ref{prop:univ-property-of-tropical-ritt} that there is a functorially induced morphism of
$\mathbf{S}$-algebras
\[
f_*\co \mathbf{S}\{x_i \: | \: i \in \Lambda_1\} \to \mathbf{S}\{y_j \: | \: j \in \Lambda_2\},
\]
and this restricts to a morphism of the basic subalgebras on the bottom.  Moreover, the congruence
$\bend \trop (\ker \alpha)$ on $S_0\{x_i \:|\: i \in \Lambda_1\}_\mathit{basic}$ is sent into the
congruence $\bend \trop (\ker \beta)$ on $S_0\{y_j \: | \: j \in \}_\mathit{basic}$ by \cite[Prop.
6.4.1]{GG1}. Hence $f_*$ descends to a morphism of quotient pairs:
\[
\begin{tikzcd}
S_1\{x_i \: | \: i \in \Lambda_1\}  \ar[d] \ar[r] & S_1\{y_j \: | \: j \in \Lambda_2\} \ar[d] \\
(S_0|S_1)\{x_i \: | \: i \in \Lambda_1\} / \bend \trop (\ker \alpha) \ar[r] & (S_0|S_1)\{y_j \: | \: j \in \Lambda_2\} / \bend \trop (\ker \beta)
\end{tikzcd}
\]
and by Proposition \ref{prop:reduction-is-left-adjoint} this induces a morphism of their reductions,
which is precisely the desired morphism
\[
\mathbf{trop}(\alpha) \to \mathbf{trop}(\beta). \qedhere
\]
\end{proof}


\subsection{The universal presentation}

Given a differential $k$-algebra $A$, consider the presentation 
\[
k\{x_a \: | \:a \in A\}
\stackrel{\mathit{Univ}}{\twoheadrightarrow} A
\]
defined by sending $x_a$ to $a$.  It takes a formal differential polynomial in the elements of $A$
and evaluates it to an element of $A$ using the differential algebra structure of $A$. A similar
morphism was studied in the non-differential setting in \cite{GG2}.  In light of the following fact,
we call this the \emph{universal presentation} of $A$.

\begin{proposition}\label{prop:universal-presentation}
The presentation $k\{x_a \:| \: a \in A\} \stackrel{\mathit{Univ}}{\twoheadrightarrow} A$ is
\begin{enumerate}
\item the final object in $\mathrm{Pres}_\mathit{lin}(A)$, and
\item the colimit of the inclusion functor $\iota\co \mathrm{Pres}^\mathit{fin}_\mathit{lin}(A) \hookrightarrow \mathrm{Pres}_\mathit{lin}(A)$.
\end{enumerate}
\end{proposition}
\begin{proof}
Part (1):  Let $\alpha\co k\{y_i \: | \: i\in \Lambda\} \twoheadrightarrow A$ be a presentation. We
will show that the set of morphisms $\Hom_{\mathrm{Pres}(A)}(\alpha,\mathit{Univ})$ contains exactly
one element that is linearly monomial.   Any morphism of presentations $f$ from $\alpha$ to
$\mathit{Univ}$ must send each variable $y_i$ to a monomial in $k\{x_a \: | \: a\in A\}$ that is
mapped to $\alpha(y_i)$ by $\mathit{Univ}$.  One option is $f(y_i) = x_{\alpha(y_i)}$, and this is
evidently the unique choice that defines a linearly monomial morphism.

Part (2):  By (1), any finite presentation $\alpha$ admits a unique linearly monomial
morphism $\alpha \to
\mathit{Univ}$, and hence there is a canonical morphism $u\co \colim \iota \to \mathit{Univ}$. Given a
finite presentation $\alpha\co k\{y_1, \ldots, y_n\} \to A$ and an element $a \in A$, we extend to a
new finite presentation $\alpha' \co k\{y_1, \ldots, y_n, y_a\} \to A$ by $y_a \mapsto a$.  The
morphism $\alpha' \to \mathit{Univ}$ sends $y_a$ to $x_a$, and thus any element $x_a$ in the universal
presentation is in the image of some finite presentation, so $u$ is surjective. 

We turn now to injectivity of $u$. If 
\[
	k\{y_1, \ldots, y_n\} \stackrel{\alpha}{\to} A \stackrel{\beta}{\leftarrow} k\{z_1, \ldots, z_m\}
\]
are two finite presentations with $\alpha(y_1) = \beta(z_1)$, then they each map to the presentation
$k\{w, y_2\ldots, y_n, z_2, \ldots, z_m\}$ by $y_1, z_1 \mapsto w$ and identity of all of the other
generators.  Hence $u$ is injective as well.
\end{proof}

\begin{example}
Consider the presentation $\alpha\co k\{y_1,y_2\} \to k\{z\} = A$ given by $\alpha(y_1) = z$ and
$\alpha(y_2)=zdz$.  In addition to the linearly monomial morphism from $\alpha$ to $\mathit{Univ}$,
there are also the monomial morphisms given by sending $y_2$ to $x_z x_{dz}$ or $x_z dx_z$.
\end{example}

\subsection{The universal tropicalization}

We define the \emph{universal tropicalization of $A$} to simply be the tropicalization of the universal presentation, $\mathbf{trop}(\mathit{Univ})$.  

\begin{theorem}
For any differential $k$-algebra $A$, there is a canonical isomorphism of pairs
\[
\mathbf{trop}(\mathit{Univ}) \cong \colim_{\alpha \in \mathrm{Pres}_\mathit{lin}(A)} \mathbf{trop}(\alpha),
\]
and if $A$ admits a finite presentation $k\{x_1, \ldots, x_n\} \to A$ then the colimit can be taken
over the subcategory $\mathrm{Pres}^\mathit{fin}_\mathit{lin}(A)$ of finite presentations.
\end{theorem}
\begin{proof}
The first statement is an immediate consequence of Proposition \ref{prop:universal-presentation}.  For the second statement, consider the canonical morphism
\[
j =(j_1, j_0)\co \colim_{\alpha \in \mathrm{Pres}_\mathit{lin}^\mathit{fin}(A)} \mathbf{trop}(\alpha) \to \mathbf{trop}(\mathit{Univ}).	
\]
Given any element $a \in A$, any finite presentation $\alpha\co k\{x_1, \ldots, x_n\} \to A$ maps to a
finite presentation $\alpha'\co k\{x_1, \ldots, x_n, x_a\} \to A$, with $\alpha'(x_a) = a$.  Hence it
follows from Proposition \ref{prop:colim-components} that $j_1$ and $j_0$ are both surjective.  We turn to
injectivity.  Observe that the congruence $\bend \trop (\ker
\mathit{Univ})$ on $S_0\{x_a \: | \: a \in A\}$  is the transitive closure of the symmetric semiring
generated by the bend relations of elements in $\trop (\ker \mathit{Univ})$, and so for any relation
$(f \sim g) \in \bend \trop (\ker \mathit{Univ})$ there exists a finite subset $\Lambda \subset A$
containing all variables appearing in either $f$ or $g$ and such that, for the restriction
$\mathit{Univ}|_\Lambda\co k\{x_a \: | \: a \in \Lambda\} \to A$, we have
\[
(f\sim g) \in \bend \trop (\ker \mathit{Univ}|_\Lambda).
\]
If $\Lambda$ does not generate $A$ as a differential algebra then we may add finitely many elements
so that it does.  We thus have a finite presentation $\beta$ such that
$(f \sim g)$ is in the image of the canonical map
\[
	\bend \trop(\ker \beta) \to \bend \trop (\ker \mathit{Univ}).
\]
Therefore the map
\[
	j_0\co \colim_{ \alpha \in \mathrm{Pres}^\mathit{fin}_\mathit{lin}} (S_0|S_1)\{x_a \: | \: a\in \Lambda\} / \bend \trop(\ker \beta) \to (S_0|S_1)\{x_a \: | \: a\in A\}/ \bend \trop(\ker \mathit{Univ})
\]
is an isomorphism.  The claim now follows from Proposition \ref{prop:colim-components}.
\end{proof}

We now come to our main result, which says that a differential algebra $A$ admits a universal differentially enhanced seminorm which is valued in the tropicalization of the universal presentation of $A$, (c.f. \cite[Theorem A]{GG2}).

\begin{theorem}\label{thm:main-result}
Given a differential $k$-algebra $A$, there is a differentially enhanced seminorm \[\mathbf{u}= (u,\widetilde{u})\co A \to \mathbf{trop}(\mathit{Univ})\] defined by sending $a\mapsto x_a$, and this is initial among differentially enhanced seminorms on $A$ compatible with the differentially enhanced seminorm $\mathbf{v}$ on $k$.
\end{theorem}

As an immediate corollary, we have:

\begin{corollary}
Let $\mathbf{T}$ be an $\mathbf{S}$-algebra.  There is a bijection
\[
	\Hom_{\mathbf{S}\mhyphen \mathrm{alg}}(\mathbf{trop}(\mathit{Univ}), \mathbf{T}) \cong \mathit{Berk}_\mathbf{T}(A).
\]
that is natural in $\mathbf{T}$.  I.e., $\mathbf{trop}(\mathit{Univ})$ co-represents the functor sending $\mathbf{T}$ to the set of differentially enhanced seminorms on $A$ taking values in $\mathbf{T}$ and compatible with $\mathbf{v}$.
\end{corollary}

The proof of the theorem requires an explicit description of the congruence $\bend \trop (\ker \mathit{Univ})$, which we provide below.

\begin{proposition}\label{prop:ker-univ-generators}
The differential ideal $\ker \mathit{Univ} \subset k\{x_a \; | \; a\in A\}$ is generated as an ideal by the following family of elements:
\begin{enumerate}
\item $x_1 - 1$;
\item $x_{\lambda a} - \lambda x_a$ for $a\in A$ and $\lambda \in k$;
\item $x_{da} - dx_a$ for $a\in A$;
\item $x_a x_b - x_{ab}$ for $a,b \in A$;	
\item $x_a + x_b + x_c$ for $a,b,c \in A$ satisfying $a+b+c = 0$.
\end{enumerate}
\end{proposition}
\begin{proof}
It is clear that all of these relations are in the kernel.  Let $f$ be an arbitrary element in the
kernel. Using relations (2), (3) and (4) we move all coefficients and differentials into the
subscripts and reduce each monomial to a single variable, so $f$ is transformed into a degree 1
non-differential polynomial $f' = \sum_{a\in \Lambda} x_a$, for some finite subset $\Lambda \subset
A$.  We can then use relation (5) (and (2) with $\lambda = -1$) repeatedly to reduce the number of
terms by replacing $x_{a_1} + x_{a_2}$ with $x_{a_1 + a_2}$, and thus we transform $f'$ to a
trinomial, which is in the kernel if and only if it is an expression of type (5).
\end{proof}

Tropicalizing the above family of elements, we have:
\begin{lemma}\label{lem:strong-tropical-basis}
The congruence $\bend \trop (\ker \mathit{Univ})$ on $S_0\{x_a \; | \; a\in A\}_\mathit{basic}$ is generated by the bend relations of the polynomials:
\begin{enumerate}
\item $x_1 \oplus 1_{S_0}$
\item $x_{\lambda a} \oplus v(\lambda) x_a$ for $a\in A$ and $\lambda \in k$;
\item $x_{da} \oplus dx_a$ for $a\in A$;
\item $x_a x_b \oplus x_{ab}$ for $a,b \in A$;	
\item $x_a \oplus x_b \oplus x_c$ for $a,b,c \in A$ satisfying $a+b+c = 0$.
\end{enumerate}
\end{lemma}
\begin{proof}
It suffices to show that the bend relations of listed expressions imply the the bend relations of any element $g\in
\trop(\ker \mathit{Univ})$.  As in the proof of Proposition \ref{prop:ker-univ-generators} above, using
the bend relations of (1)--(4) allows us reduce $g$ to an expression of the form 
\[
\bigoplus_{a \in \Lambda} x_a, \text{ with } \sum_{a\in \Lambda} a = 0.
\]
Let us call a finite set $\Lambda \subset A$ a \emph{null set} if $\sum_\Lambda a = 0.$ It remains to show that the bend relations of sums over null sets of size 3 (i.e., relation (5) from the list) imply the bend relations for sums as above over null sets of arbitrary size.

We prove this by induction on the cardinality $n$ of the null set $\Lambda$. The base case $n=3$
is simply relation (5).  Assume the bend relations hold for all sums over null sets of size $\leq
n$, and consider a null set $\Lambda = \{a_1, \ldots, a_{n+1}\}$.  Let 
\[
b = a_1 + a_2 = -(a_3 + \cdots + a_{n+1}),
\]
so we have null sets $\Lambda_1 = \{a_1, a_2, -b\}$ and $\Lambda_2 = \{b, a_3, \ldots, a_{n+1}\}$ of size 3 and $n$.  Then
\begin{align*}
	x_{a_1} &\oplus x_{a_2} \oplus \cdots \oplus x_{a_{n+1}} \quad & \\
	&\sim x_{a_1} \oplus x_{a_2} \oplus \cdots \oplus x_{a_{n+1}} \oplus x_{-b} \quad &\text{(using the bend relations of $\Lambda_1$ to pull out $x_{-b}$)}\\
	&\sim  x_{a_2} \oplus \cdots \oplus x_{a_{n+1}}\oplus x_{-b} \quad &\text{(using the bend relations of $\Lambda_1$ to delete $x_{a_1}$)} \\
	&\sim x_{a_2} \oplus \cdots \oplus x_{a_{n+1}} \oplus x_{b} \quad &\text{(since $x_b \sim x_{-b}$ by (2))}\\
	&\sim  x_{a_2} \oplus \cdots \oplus x_{a_{n+1}} \quad &\text{(using the bend relations of $\Lambda_2$ to delete $x_{b}$).}
\end{align*}
Since $a_1$ was chosen arbitrarily, this shows that the bend relations of sums over null sets of size $n+1$ hold.
\end{proof}
Note that relations (1), (2), (4) and (5) correspond to the four conditions defining a non-archimedean seminorm in Definition \ref{def:generalized-valuations}.

Let us write $U_1 \stackrel{\pi}{\to} U_0$ for the pair $\mathbf{trop}(\mathit{Univ})$, which is the reduction of the pair \[S_1\{x_a \: | \: a\in A\} \to (S_0 | S_1)\{x_a \: | \: a\in A\} / \bend \trop( \ker \mathit{Univ}).\]  

\begin{lemma}\label{lem:commutes-with-d}
The relation $d(x_a) = x_{da}$ holds in $U_1$.
\end{lemma}
\begin{proof}
Observe that for any $a\in A$ and any $n\in \mathbb{N}$ the elements $d^nx_a$ and $x_{d^na}$ in
$U_1$ are mapped to the elements in $U_0$ that are equivalent modulo the congruence $\bend
\trop(\ker \mathit{Univ})$ by Lemma \ref{lem:strong-tropical-basis} relation (3).  Hence the
relation $d(x_a) \sim x_{da}$ is in the reduction congruence and so it holds in $U_1$.
\end{proof}

\begin{proof}[Proof of Theorem \ref{thm:main-result}]
It is clear from the definition that $\pi \circ \widetilde{u} = u$.  By Lemma
\ref{lem:commutes-with-d}, the map $\widetilde{u}\co A \to U_1$ commutes with the differential, and
by relations (1), (4), and (5) of Lemma \ref{lem:strong-tropical-basis}, the map $u\co A \to U_0$ is
a non-archimedean seminorm.  Thus $\mathbf{u}$ is a differentially enhanced
seminorm.  Moreover, relation (2) implies that $u$ is compatible with the seminorm $v\co A \to S_0$.

It remains to show that there is a unique morphism from $\mathbf{u}$ to any other differentially
enhanced seminorm $\mathbf{w}=(w,\widetilde{w})\co A \to \mathbf{T}$ compatible with $\mathbf{v}\co A
\to \mathbf{S}$.  Let $\widetilde{y}_a = \widetilde{w}(a) \in T_1$.  Mapping these elements and their
derivatives down to $T_0$ gives a list of elements  $y^{(j)}_a=w(d^j a)  \in T_0$ for $a\in A$ and
$j\in \mathbb{N}$.  Since $w$ is a non-archimedean seminorm, it follows from Lemma \ref{lem:strong-tropical-basis}
that the semiring homomorphism
\[
S_0\{x_a \:| \: a\in A\}_{\mathit{basic}}  \to T_0
\]
 sending $d^j x_a \mapsto y^{(j)}_a$ descends to the quotient by the congruence $\bend\trop(\ker \mathit{Univ})$.
Hence, by Proposition \ref{prop:quotient-property}, there is a unique morphism of $\mathbf{S}$-algebras $f_\mathbf{w}\co
 \mathbf{trop}(\mathit{Univ}) \to \mathbf{T}$ sending $x_a$ to $\widetilde{y}_a$ on top, and to
 $y_a$ on the bottom.  By construction, composition with $f_\mathbf{w}$ sends $\mathbf{u}$ to
 $\mathbf{w}$, and this completes the proof.
\end{proof}

\end{document}